\documentclass[12pt]{amsart}
\usepackage{mathabx}
\usepackage{extpfeil}
\usepackage{hyperref}
\usepackage{amssymb, fancyhdr, stmaryrd, amsthm, mathrsfs, textcomp}
\DeclareMathAlphabet{\mathpzc}{OT1}{pzc}{m}{it}
\usepackage[utf8]{inputenc}
\usepackage{amsfonts,amsmath,amscd,amssymb}
\usepackage{dsfont,mathtools}
\usepackage{appendix}
\usepackage{xspace}
\usepackage{xcolor, todonotes}
\usepackage{tikz-cd}
\usepackage{listings}
\usepackage{faktor}
\usepackage{calligra}
\usepackage{graphicx}

\def\bA{\ensuremath\mathbb{A}}
\def\bB{\ensuremath\mathbb{B}}

\def\bC{\ensuremath\mathbb{C}}

\def\bF{\ensuremath\mathbb{F}}

\def\bI{\ensuremath\mathbb{I}}
\def\bJ{\ensuremath\mathbb{J}}
\def\K{\ensuremath\mathbb{K}}
\def\N{\ensuremath\mathbb{N}}

\def\bX{\ensuremath\mathbb{X}}
\def\Z{\ensuremath\mathbb{Z}}

\def\x{\ensuremath\times}

\def\to{\ensuremath\rightarrow}
\def\sset{\ensuremath\subseteq}

\def\<{\ensuremath\langle}
\def\>{\ensuremath\rangle}

\DeclareMathOperator{\Id}{Id}

\DeclareMathOperator{\coeq}{coeq}
\DeclareMathOperator{\eq}{eq}
\DeclareMathOperator{\im}{im}

\DeclareMathOperator{\Hom}{Hom}
\DeclareMathOperator{\Cohom}{Cohom}
\DeclareMathOperator{\Cnt}{Const}

\DeclareMathOperator{\SC}{SC}

\newtheorem{result}{Theorem}
\newtheorem{assum}{Assumption}
\newtheorem{thm}{Theorem}[section]
\newtheorem{lemma}[thm]{Lemma}
\newtheorem{prop}[thm]{Proposition}

\newtheorem{cor}[thm]{Corollary}

\theoremstyle{definition}
\newtheorem{defn}[thm]{Definition}
\newtheorem{conj}[thm]{Conjecture}
\newtheorem{prob}[thm]{Problem}

\newtheorem{question}[thm]{Question}

\newcommand{\vi}{\ensuremath{\mathbf{i}}\xspace}

\newcommand{\vk}{\ensuremath{\mathbf{k}}\xspace}

\newcommand{\vu}{u}
\newcommand{\vv}{v}
\newcommand{\vw}{w}
\newcommand{\vx}{x}
\newcommand{\vy}{y}
\newcommand{\vz}{z}

\newcommand{\sA}{\mathcal{A}}
\newcommand{\sB}{\mathcal{B}}
\newcommand{\sC}{\mathcal{C}}

\newcommand{\sD}{\mathcal{D}}

\newcommand{\sJ}{\mathcal{J}}
\newcommand{\sK}{\mathcal{K}}
\newcommand{\sM}{\mathcal{M}}

\newcommand{\sN}{\mathcal{N}}

\newcommand{\sS}{\mathcal{S}}
\newcommand{\ssS}{\mathscr{S}}

\newcommand{\sT}{\mathcal{T}}

\newcommand{\sW}{\mathcal{W}}

\newcommand{\sInd}{\mathcal{I}\!\mathpzc{n}\mathpzc{d}}
\newcommand{\sCoi}{\mathcal{C}\!\mathpzc{o}\mathpzc{i}\mathpzc{n}\mathpzc{d}}
\newcommand{\sRes}{\mathcal{R}\!\mathpzc{e}\mathpzc{s}}
\newcommand{\sSets}{\mathcal{S}\!\mathpzc{e}\mathpzc{t}\mathpzc{s}}
\newcommand{\KCh}{{\mathbb K}{\mathrm{-}}{\mathcal{C}}\!\mathpzc{h}}

\newcommand{\Fib}{{\mathbb F}}
\newcommand{\Cof}{{\mathbb C}}
\newcommand{\WEq}{{\mathbb W}}
\newcommand{\Cyl}{{\mathrm C}{\mathrm y}{\mathrm l}}

\newcommand{\clm}{{\varinjlim}}

\newcommand{\HO}{{\mathrm H}{\mathrm o}}

\newcommand{\Map}{{\mathcal M}{\mathrm a}{\mathrm p}}

\newcommand{\unit}{I} 
\newcommand{\chf}{C}

\newcommand{\LB}{{\mathrm L}{\mathrm B}{\mathrm L}}
\newcommand{\RB}{{\mathrm R}{\mathrm B}{\mathrm L}}

\begin{document}
\title[Co-Contra Correspondence]{General Comodule-Contramodule Correspondence\footnote{In memory of Sasha Ananin}}
\author{Katerina Hristova}
\email{katherina.hristova@gmail.com}
\address{School of Mathematics, University of East Anglia, Norwich, NR4 7TJ, UK}
  \author{John Jones}
\email{jdsjones200@gmail.com}
\address{Department of Mathematics, University of Warwick, Coventry, CV4 7AL, UK}
    \author{Dmitriy Rumynin}
\email{D.Rumynin@warwick.ac.uk}
\address{Department of Mathematics, University of Warwick, Coventry, CV4 7AL, UK}

\date{7 November 2022}
\subjclass[2010]{Primary 18D20, Secondary 18N40, 16T15, 55U40}
\keywords{symmetric monoidal category, comodule,  comonad,  contramodule,  model category, monad, Quillen adjunction}
\thanks{We would like to thank Tomasz Brzezi\'{n}ski, Emanuele Dotto, Kathryn Hess, Leonid Positselski and Robert Wisbauer for their interest in our work and sharing helpful information.}

\begin{abstract}
  This paper is a fundamental study of comodules and contramodules over a comonoid in a symmetric closed monoidal category. We study both algebraic and homotopical aspects of them. Algebraically, we enrich the comodule and contramodule categories over the original category, construct enriched functors between them and enriched adjunctions between the functors.
Homotopically, for simplicial sets and topological spaces, we investigate the categories of comodules and contramodules and the relations between them.
\end{abstract}

\maketitle


\section{Introduction}
\subsection{Comodules and contramodules}
Let $\K$ be a field, $C$ a coalgebra over $\K$.  Let $\Delta : C \to C \otimes C$ be the comultiplication, $\varepsilon : C \to \K$ the counit.  A right comodule over $C$ is a vector space $X$ over $\K$ equipped with a structure map $\rho : X \to X \otimes C$.  This structure map must satisfy the natural coassociativity and counitality conditions.
Coassociativity is the requirement that the diagram 
$$
\begin{CD}
 X @>{\rho}>> X \otimes C \\
 @V{\rho}VV  @VV{\rho \, \otimes 1} V\\
 X \otimes C @>{1 \otimes \, \Delta}>> X \otimes C \otimes C\\
 \end{CD}
 $$
commutes and counitality is the requirement that the composite
$$
\begin{CD} 
X @>{\rho}>>X \otimes C @>{1\otimes \, \varepsilon}>>  X \otimes \K @>{\cong}>> X
\end{CD}
$$
is the identity.

At the end of their paper \cite{EiM2} Eilenberg and Moore point out that there is another natural kind of ``module'' over $C$ which they call a {\em contramodule}.  A contramodule over $C$ is a vector space $Y$ equipped with a structure map $\theta : \Hom_{\K}(C, Y) \to Y$ satisfying the following contraassociativity and contraunitality conditions.  Contraassociativity is the requirement that the diagram 
$$
\begin{CD}
 \Hom_{\K}(C, \Hom_{\K}(C,Y)) @>{\theta_*}>> \Hom_{\K}(C,Y) @>{\theta}>> Y \\
 @V{=}VV                                                     @.                                 @VV{=}V \\
 \Hom_{\K}(C \otimes C, Y) @>{\Delta^*}>> \Hom_{\K}(C,Y)@>{\theta}>> Y
 \end{CD}
$$
commutes and contraunitality is the requirement that the composite
$$
\begin{CD} 
Y = \Hom_{\K}(\K,Y) @>{\varepsilon^*}>>  \Hom_{\K}(C, Y) @>{\theta}>> Y
\end{CD}
$$
is the identity.

Since $C^*\otimes Y$ is naturally a subspace of $\Hom_{\K}(C,Y)$, the structure map $\theta$ makes $Y$ a module over the algebra $C^*$.
Similarly, a comodule $X$ gets a $C^\ast$-module structure from the composition
$$
\begin{CD} 
X \otimes C^\ast @>{\rho\otimes 1}>>  X \otimes C \otimes C^\ast @>{1\otimes ev}>>  X\otimes \K = X\, .
\end{CD}$$
If $C$ is finite dimensional, then $\Hom_{\K}(C,Y) = C^*\otimes Y$ and there is no difference between comodules, $C^\ast$-modules and contramodules.
This is the simplest example of the comodule-contramodule correspondence.

The theory of contramodules over a coalgebra was completely neglected until the early 2000's when Positselski, motivated by his work on the theory of semi-infinite cohomology in the geometric Langlands program, took up the study of contramodules.  Much of his work was published in 2010 \cite{Pos2}. Positselski showed how to set up the co-contra correspondence for coalgebras without any finite dimensionality hypotheses.   This correspondence is not a one-to-one correspondence, but Positselski was able to prove that it defines an equivalence between appropriate derived categories of comodules and contramodules.  In \cite{Bohm} Bohm, Brzezi\'{n}ski, and Wisbauer gave a clean account of the theory of comodules and contramodules in categories of modules in terms of monads, comonads, and adjoint functors.  This account fits in well with the original work of Eilenberg and Moore.

The starting point for our work is that the ingredients for the comodule-contramodule correspondence are present in many interesting examples which are not module categories.  The simplest (and perhaps the most fundamental of all examples) is the category of sets.  Other examples, which we study in some detail in this paper, include the category of chain complexes, simplicial sets and topological spaces.
Further examples, which we do not study, include the categories of spectra,  $G$-sets where $G$ is a discrete group, and $G$-spaces where $G$ is a topological group.   

Let $\sC$ be a closed symmetric monoidal category.  Such a category $\sC$ comes equipped with a symmetric tensor product bifunctor $\otimes : \sC \times \sC  \to \sC$
satisfying the usual associativity and unitality conditions
and an internal hom bifunctor $[\, , \, ] : \sC^{op} \otimes \sC  \to \sC$
such that 
for each pair of objects $A, B\in \sC$ there is a  natural isomorphism of functors of $X$ 
$$
\sC(X \otimes A, B) \to \sC(X, [A,B])).
$$
Here $\sC(A,B)$ denotes the set of morphisms in $\sC$ with source $A$ and target $B$, while $[A,B]$ denotes the {\em internal hom object}.
Let $\unit$ be the monoidal unit in $\sC$.
The internal hom object determines the hom set:
$$
\sC(\unit, [A,B]) = \sC(A,B).
$$

The diagrams, which define a coalgebra and its comodules in the category of vector spaces, use only the tensor functor $\otimes$.   So they make sense in $\sC$ and define the notion of a comonoid $C$ in $\sC$ and its comodules.
Since $\sC$ is closed, contramodules over $C$ are defined by using the formal analogues of the diagrams in the category of vector spaces which define contramodules with $\Hom_{\K}(C, Y)$ replaced by the internal hom object $[C,Y]$.
Therefore, we have comonoids, comodules and contramodules in any closed symmetric monoidal category. The aim of this paper is to develop the co-contra correspondence in this general context.

\subsection{The co-contra correspondence}  
Let $C$ be a comonoid in $\sC$.
Now we consider the categories of comodules and contramodules over $C$.
The definition of a morphism of comodules (or contramodules) over $C$ is a morphism (in $\sC$) of their underlying objects such that the obvious diagrams commute.
This defines two perfectly good categories: the category $\sC_C$ of comodules over $C$ and the category $\sC^C$ of contramodules over $C$.
If we are working with vector spaces over a field, these sets of morphisms have a natural vector space structure and all is well.
But in $\sC$ we can only get sets of morphisms in this way.
What we really want is to give the categories $\sC_C$ and $\sC^C$ the structure of categories enriched in $\sC$.
The theory of enriched categories is quite subtle; the standard reference is \cite{Kelly}.

In other words, this means that for any two comodules $X, X'$ over $C$ we must construct a hom object $[X,X']_C\in \sC$, and for any two contramodules $Y,Y'$ over $C$ we must construct a hom object $[Y,Y']^C \in \sC$.
The obvious idea is to define $[X,X']_C$ as the equaliser of two morphisms in $\sC$ from $[X,X'] \to [X,X'\otimes C]$.
This equaliser must exist: the categories $\sC_C$ and $\sC^C$ have the  structure of enriched categories over $\sC$
if $\sC$ satisfies the following completeness property (see Proposition~\ref{C_TF enriched}):
\begin{assum} \label{assum_1}
Each pair of morphisms $X \rightrightarrows Y$
with a common left inverse admits an equaliser.
\end{assum}

We will also need the dual assumption:
\begin{assum} \label{assum_2}
Each pair of morphisms $X \rightrightarrows Y$
with a common right inverse admits a  coequaliser.
\end{assum}

Our first objective is to establish the next theorem, 
the general comodule-contramodule correspondence.

\begin{result} {\em (Theorem~\ref{adjoint_LR})}
  Suppose a closed symmetric monoidal category $\sC$ satisfies Assumptions~\ref{assum_1} and \ref{assum_2}.
  Then there is an enriched adjoint pair 
  of enriched  functors
$$(L \dashv R), \ \ \ L : \sC^C \rightleftarrows \sC_C : R \, .$$
\end{result}

We devote Chapter~\ref{s1} to describing the construction of $L \dashv R$ but we postpone the proofs to Chapter~\ref{s_appendix}.
In Chapter~\ref{s2e} we explain how the construction in Chapter~\ref{s1} works in three examples: the category of chain complexes over a field, the category of sets, and the category of simplicial sets.  

Let us discuss this result in $\sSets$, the category of sets.
The monoidal product is the cartesian product of sets and the internal hom set is the set of functions.
It is easy to see that any set $C$ has a unique structure of a comonoid.  The comultiplication is the diagonal map. A simple argument with unitality (cf. Section~\ref{comodules_sets}) shows that the diagonal map is, indeed, the unique comonoid structure on $C$.  

If $X$ is a set, then a bijection 
$$
\coprod_{c \,\in C} U_c \to X
$$
defines the structure of a comodule over $C$ on $X$ (cf. Section~\ref{comodules_sets}).  If $Y$ is a set, then a bijection
$$
Y \to \prod _{c\, \in C} V_c
$$
defines the structure of a contramodule over $C$ on $Y$ (cf. Section~\ref{RX}).

\begin{result} {\em (Theorems~\ref{any_contraset} and \ref{set_CCC})}
Let $C$ be a set considered as a comonoid in the category of sets.
\begin{enumerate}
\item
Every comodule over $C$ is isomorphic to $\coprod_{c \in C} U_c$, where $U_c$ is a collection of sets parametrised by $C$.
\item
Every contramodule over $C$ is isomorphic to $\prod_{c \in C} V_c$, where $V_c$ is a collection of sets parametrised by $C$.
\item
The functors $L : \sSets^C \to \sSets_C$ and  $R : \sSets_C \to \sSets^C$ are characterised by
$$
L\left(\prod_{c \in C} V_c\right) = \coprod_{c \in C} V_c, \quad R : \left(\coprod_{c \in C} U_c\right) = \prod_{c \in C} U_c
$$
where all sets $V_c$ must be non-empty.
\end{enumerate}
\end{result}

\subsection{Homotopy theory in categories of comodules and contramodules}
The motivation for introducing homotopy theory is the main theorem of Positselski.
It states that in the algebraic context of comodules and contramodules  over a DG-coalgebra the co-contra correspondence defines an equivalence between the coderived category of comodules and the contraderived category of contramodules.  It is natural to think about this theorem in terms of Quillen's model categories.

A model category is a category $\sM$ together with three distinguished classes of morphisms: {\em cofibrations}, {\em fibrations} and {\em weak equivalences}, satisfying appropriate axioms.
If $\sM$, $\sN$ are model categories, a Quillen adjunction between them is a pair of adjoint functors
$$(A \dashv B), \ \ \ A : \sN \rightleftarrows \sM : B\, ,$$
satisfying certain axioms. Further axioms turn a Quillen adjunction into a Quillen equivalence.

A {\em symmetric monoidal model category} is the natural notion of a category with a compatible symmetric monoidal structure and model structure (\cite{Hir,Hov} for further details).
Let $\sC$ be a symmetric monoidal model category and let $C$ be a comonoid in $\sC$.
Form the categories $\sC_{C}$ and $\sC^{C}$.
There are forgetful functors $G_C:\sC_{C} \to \sC$ and $G^C: \sC^{C} \to \sC$.
Under suitable conditions, we can use these functors to define induced model structures on $\sC_{C}$ and $\sC^{C}$.
With more conditions we can show that the functors $L : \sC^C \rightleftarrows \sC_C : R$ define a Quillen adjunction.
With yet more conditions we can adjust the model structures of $\sC_C$ and $\sC^C$, by a technique known as Bousfield localisation, so that with the new model structures the functors
$L : \sC^C \rightleftarrows \sC_C : R$ become a Quillen equivalence.  This leads to the following two results.

\begin{result}  {\em (Proposition~\ref{Quillen_adjunction_cartesian_cat})}
  Suppose that the closed symmetric monoidal model category $\sC$ is cartesian closed.
  If the left-induced model structure exists on $\sC_C$ and the right-induced model structure exists on $\sC^C$,   
then the pair $({L}\dashv R)$ is a Quillen adjunction.  
\end{result}

\begin{result} \label{res4}  {\em (Theorem~\ref{Quillen_equivalence_cartesian_cat})}
Suppose that $\sC$ satisfies the following assumptions.
\begin{enumerate}
\item
$\sC$ is a locally presentable category,
\item
$\sC$ is a  cartesian closed symmetric monoidal model category, 
\item
$\sC$ is a left and right proper model category.
\end{enumerate}
Let $C$ be 
a comonoid in $\sC$.  Then there exist a left Bousfield localisation $\LB(\sC^C)$ and
a right Bousfield localisation $\RB(\sC_C)$ such that the functors
$$L : \LB(\sC^C) \rightleftarrows \RB(\sC_{C}) : R$$
form a Quillen equivalence.
\end{result}

The last theorem should be interpreted to mean, as hinted in the previous paragraph, that under certain categorical assumptions we can find reasonable model structures on $\sC^C$ and $\sC_C$
so that the functors $(L\dashv R)$ define a Quillen equivalence (see Chapter~\ref{s2}). 
As a specific example, the category of simplical sets satisfies all the conditions of this theorem (Theorem~\ref{simplicial_set_CCC_strong}).

\subsection{Comodules and contramodules in the category of topological spaces} 
In Chapter~\ref{s6}  the base category is the category $\sW$ of compactly generated, weakly Hausdorff spaces, the most standard convenient category of topological spaces. A comonoid in $\sW$ is a topological space $C$ with comultiplication given by its diagonal map. Most of the chapter is devoted to the general study of comodules and contramodules in $\sW$. One non-obvious fact about this category is the following theorem. 

\begin{result}{\em (Theorem~\ref{TB_cocomplete})}
Let $C$ be a topological space considered as a comonoid in $\sW$. Then the category of contramodules $\sW^C$ is cocomplete.
\end{result}

The conditions of Theorem~\ref{res4} do not hold in $\sW$ for set-theoretic reasons. Yet we can prove some interesting facts about the topological comodule-contramodule correspondence.

\begin{result} {\em (Propositions~\ref{top_CCC_Grothendieck}, \ref{easy_top_fact}  and Theorem~\ref{last_thm})} 
Let $C$ be a topological space considered as a comonoid in $\sW$. 
\begin{enumerate}
\item The co-contra correspondence $L : \sW^C \to \sW_C : R$ is a Quillen adjunction between $\sW^C$ and $\sW_C$.
 \item If all topological spaces are subsets of a Grothendieck universe, the adjunction $L : \sW^C \to \sW_C : R$ defines a Quillen equivalence between a left Bousfield localisation $\LB(\sW_C)$ and a right Bousfield localisation $\RB(\sW^{C})$.
\item If $X,Y\in \sW_C$ are CW-complexes and $f \in\sW_C(X,Y)$ is a weak equivalence, then $R(f)$ is a  weak equivalence.
\item Suppose that $C$ is a CW-complex of finite type. If $X,Y\in \sW_C$ are fibrant and $f \in \sW_C(X,Y)$ is a weak equivalence such that $\pi_0 ( f)$ is an isomorphism, then $R(f)$ is a  weak equivalence.
\end{enumerate}
\end{result}

\section{Monad-Comonads Adjoint Pairs over Closed Categories}
\label{s1}


\subsection{Closed categories}
\label{s1.1}
Let us consider a symmetric monoidal category ${\sC}$
with hom sets $\sC (X,Y)$, tensor product $\otimes$, unit object $\unit$, associator $\alpha$,
symmetric braiding $\gamma$, left unitor $\lambda$ and right unitor $\varpi$. The latter four are natural isomorphisms
$$
\alpha_{X,Y,Z}: (X\otimes Y)\otimes Z \xrightarrow{\cong} X \otimes (Y\otimes Z), \
\gamma_{X,Y}: X\otimes Y \xrightarrow{\cong} Y\otimes X,
$$
$$
\lambda_X : \unit \otimes X \xrightarrow{\cong} X, \ 
\varpi_X : X \otimes \unit \xrightarrow{\cong} X,
$$
depending on objects $X,Y,Z \in \sC$.

The category $\sC$ is called {\em a closed symmetric monoidal category} if for any object $X\in {\sC}$
the endofunctor $- \otimes X$ 
admits a right adjoint
endofunctor $[X, - ]_{\sC}$ called {\em the internal hom} \cite{LaP}.
When the category in question is clear, we use the shorthand notation  $[X,Y]$  for $[X,Y]_{\sC}$.

Recall {\em the functor of global sections}:
\begin{equation} \label{global_sections}
\Gamma : \sC \rightarrow \sSets, \ 
\
\Gamma (X) \coloneqq \sC (\unit, X).
\end{equation}
The relation between the hom and the internal hom is a natural isomorphism 
\begin{equation} 
\sC(X,Y) \cong  \Gamma ([X, Y]).
\end{equation}

In general, $[X,Y]$ is not even a set. A good category to keep in mind for illustration purposes is the category of $G\mbox{-}\sSets$ of $G$-sets for a group $G$. 
This category is {\em cartesian}: $X\otimes Y$ is the product $X\times Y$. The internal hom $[X,Y]$ is the set of all the functions $X\rightarrow Y$.
The ordinary hom is its fixed point set: $G\mbox{-}\sSets (X,Y) = [X,Y]^G$.

\subsection{Enriched categories}
\label{s1.enr}
          The standard reference for enriched categories is Kelly's book \cite{Kelly}.  
A category $\sA$ enriched in $\sC$ has 
hom objects and compositions
\begin{equation}
\label{hom_compl}  
[X,Y]_\sA\in\sC, \ \   
c_{X,Y,Z} \in \sC([Y,Z]_\sA \otimes [X,Y]_\sA , [X,Z]_\sA)
\end{equation}
for all $X,Y,Z\in\sA$, 
satisfying the standard axioms.
It  can be turned into an ordinary category by setting
\begin{equation} \label{enrich}
  \sA (X,Y) \coloneqq \Gamma ([X,Y]_\sA)\, .
\end{equation}
In the opposite direction, 
{\em an enrichment} of a category $\sA$ is a structure of enriched category such that
\eqref{enrich} holds.

A closed symmetric monoidal category $\sC$ is enriched in itself.
Its opposite category  $\sC^{op}$ is enriched in $\sC$:
\begin{equation}
[X,Y]_\sC \coloneqq [X,Y]\, , \ \ \ 
[X,Y]_{\sC^{op}} \coloneqq [Y,X]\, .
\end{equation}

For categories $\sA, \sB$ enriched in $\sC$, a {\em $\sC$-enriched functor}
$H: \sA \to \sB$ consists of the following data, satisfying the standard axioms:
\begin{itemize}
\item a map $H: \sA \to \sB$ from the objects of $\sA$ to the objects of $\sB$,
\item an $\sA \x \sA$-indexed family of morphisms in $\sC$
  \begin{equation} \label{enriched_functor}
    H_{X,Y}: [X,Y]_\sA \to [H X, H Y]_\sB,\end{equation}
which respect the enriched composition and units in $\sA$ and $\sB$.
\end{itemize}

\subsection{Adjoint functors}
\label{s1.1a}
Fix a closed symmetric monoidal category $\sC$.
Consider a pair of endofunctors  $T,F : \sC\rightarrow \sC$, not necessarily enriched. There are three different notions of adjointness:
\begin{itemize}
\item 
If a natural isomorphism of bifunctors
$$\sC (T - , {-}), \sC (- , F {-}):
{\sC}^{op} \times \sC \rightarrow \sSets$$
is chosen, then $T$ and $F$ are {\em externally adjoint}.
\item 
If $T$ and $F$ are enriched and a natural isomorphism of bifunctors
\begin{equation}\label{ext_adjunction}
  [T-,{-}], [-,F{-}]: {\sC}^{op}\otimes \sC \rightarrow \sC
\end{equation}
is chosen, 
then $T$ and $F$ are {\em internally adjoint}.
\item 
  Further, if 
  the natural isomorphism \eqref{ext_adjunction}
  is enriched, 
then $T$ and $F$ are {\em enriched adjoint}.
\end{itemize}
Without standard notation to distinguish the three, we write $(T\dashv F)$ in all of them.
These notions are related. 

\begin{lemma} \label{internal_to_external}
An enriched adjoint pair of endofunctors is internally adjoint.
An internally adjoint pair of endofunctors is externally adjoint.
\end{lemma}
\begin{proof}
The second claim is obvious: just forget the enrichment.
The first claim follows from applying the global sections~\eqref{global_sections} $\Gamma$ to the internal adjunction $\tau$
$$\sC (T-,{-}) \cong \Gamma ([T-,{-}]) \stackrel{\Gamma (\tau)}{\cong} \Gamma ([-,F{-}])  \cong \sC (-,F{-})$$
as functors 
${\sC}^{op}\times \sC \rightarrow \sSets$.
\end{proof}

\begin{defn} \label{defn_chief}
Let $(T\dashv F)$ be an internally adjoint pair of endofunctors on $\sC$. We define {\em the chief} (or {\em the chief object})
of the pair $(T \dashv F)$ as $\chf \coloneqq T\unit$.
\end{defn}

The following lemma,
motivating our interest in the chief,
is {\em surprising}, due to its implications. 
\begin{lemma} \label{surprising}
Let $(T\dashv F)$ be an internally adjoint pair of endofunctors of $\sC$ and 
  $\chf$ their chief. 
  Then there are natural isomorphisms of functors
  $$
  F\cong [\chf, - ], \ \ 
  T \cong \ - \otimes \chf
  $$
  such that the following diagram commutes   
  for all $X,Y\in\sC$: 
  $$
\begin{CD}
 [T(X),Y] @>{{\cong}_{1}}>> [X, F(Y)]  \\
 @V{{\cong}_{4}}VV  @VV{{\cong}_{2}} V\\
 [X \otimes C,Y] @>{{\cong}_{3}}>> [X, [C,Y]]\\
 \end{CD} 
 $$
\end{lemma}
\begin{proof} Using the isomorphism $i_X : X \rightarrow [\unit,X]$,
  we obtain the first natural isomorphism as the composition
$$
[\chf,X] \xrightarrow{\cong} [T\unit , X] \xrightarrow{\cong} [\unit , F X] \xrightarrow{\cong} FX. 
$$
Thus, we have natural isomorphisms of bifunctors $\cong_{1}$, $\cong_{2}$ and $\cong_{3}$.
We define $\cong_{4}$ as the composition
\[
[X\otimes \chf , Y] \xrightarrow{\cong}
[X, [\chf,Y]] \xrightarrow{\cong}
[X,FY] \xrightarrow{\cong}
[TX,Y].
\]  
This ensures commutativity of the square. It remains to notice that
the natural isomorphism of representable functors
$$
\gamma_X : \sC(X\otimes \chf, - ) = \Gamma ([X\otimes \chf, - ]) \xrightarrow{\cong} \Gamma ([TX, - ])  = \sC (TX, - )
$$
yields, by the Yoneda Lemma,  an isomorphism of representing objects
$$
\beta_X : X\otimes \chf \xrightarrow{\cong}   TX, 
$$
natural in $X$.
Hence, $\beta_X$ is a natural
isomorphism of functors.
\end{proof}

Since  $[\chf, - ]$ and $\, - \, \otimes \chf$ are enriched adjoint, the surprising lemma
allows us to replace an internal adjunction with an enriched (possibly different) adjunction.

\begin{cor}
Let $(T\dashv F)$ be an internally adjoint pair of enriched endofunctors of $\sC$.
There exists an enriched adjunction $(T\dashv F)$.
\end{cor}

\subsection{Monads and comonads}
\label{s1.2}
Consider an object $C$ of a monoidal category $\sC$
and the corresponding enriched adjoint pair $(T \dashv F)$ of endofunctors $T=\, - \, \otimes \chf$ and $F=[\chf, - ]$.
It is well known that 
$$
T \mbox{ is a monad } \stackrel{\mbox{\cite[{\small Prop. 3.1}]{EiM}}}{\Longleftrightarrow}
F \mbox{ is a comonad } \Leftrightarrow
\chf \mbox{ is a monoid.}
$$
Our goal is to make a precise dual enriched statement to this one.
Consider a monad  $(F, \mu, \eta)$ and a comonad $(T, \delta, \epsilon)$. Here
\[ \mu: FF \longrightarrow F, \
\eta: \Id_{\sC} \longrightarrow F, \
\delta: T \longrightarrow TT, \ 
\epsilon: T \longrightarrow \Id_{\sC}\]
are natural transformations, satisfying associativity and unitality conditions 
\cite[2.3, 2.4]{Bohm}, \cite[§2]{EiM}, \cite[Sec. II]{Str_2}. 
See \cite{nlab_SM} for what makes a (co)monad strong or enriched.

\begin{lemma} 
\label{como_mo}
Let $\sC$ be a closed symmetric monoidal category,  $\chf\in\sC$.
Consider the enriched endofunctors $T=\, - \, \otimes \chf$ and $F=[\chf, - ]$,
together with 
their enriched adjunction $(T \dashv F)$.
There are bijections between the following three sets
\begin{itemize}
\item the set of strong comonad structures on $T$,
\item the set of strong monad structures on $F$,
\item the set of comonoid structures on $C$.
\end{itemize}
\end{lemma}
\begin{proof}
  Start with a comonad structure on $T$.
  We get a comultiplication and a counit on $\chf$ by
  \[
  \chf \xrightarrow{\cong} \unit \otimes \chf \xrightarrow{\delta_{\unit}} \chf \otimes \chf
  \mbox{ and }
  \chf \xrightarrow{\cong} \unit \otimes \chf \xrightarrow{\epsilon_{\unit}} \unit \, .
  \]
  Verification of the axioms is routine.
  

If $(\chf, \delta, \varepsilon)$ is a comonoid, we obtain a strong comonad structure on $T$ by defining the natural transformations $\delta$, $\epsilon$ explicitly:
\[\delta_X: TX=X \otimes \chf \xrightarrow{\Id_X \otimes \, \delta} X \otimes (\chf \otimes \chf) \xrightarrow{\alpha^{-1}} TTX,\]
\[\epsilon_X: TX \xrightarrow{\Id_X \otimes \, \varepsilon}  X \otimes I \cong X . \]
Again, all the axioms are routine.

This gives the bijection between
the set of comonoid structures on $\chf$ to the class of strong comonad structures on $T$.
In particular, this class is a set.

A proof for the set of strong monad structures is similar.
\end{proof}

\subsection{Accessibility and presentability}
\label{s1.2a}
Occasionally we assume that $\sC$ is {\em accessible} or {\em locally presentable}.
We follow Ad\'{a}mek and Rosicky \cite{AdRo} with our terminology.

For the convenience of the reader, we recall the key definitions. Given a regular cardinal $\Lambda$, an object $X$ of some category $\sB$ is called {\em $\Lambda$-presentable}, if  $\sB(X,-)$ preserves $\Lambda$-directed colimits.
An object $X$ is {\em presentable}, if it is $\Lambda$-presentable for some regular cardinal $\Lambda$. 

The category $\sB$ is {\em locally $\Lambda$-presentable}, if it is cocomplete and admits a set $Z$ of $\Lambda$-presentable objects such that every object is a $\Lambda$-directed colimit of objects from $Z$.
The category $\sB$ is {\em locally presentable}, if it is {\em locally $\Lambda$-presentable} for some regular cardinal $\Lambda$.

Similarly, the category $\sB$ is {\em $\Lambda$-accessible}, if it has $\Lambda$-directed colimits and admits a set $Z$ of $\Lambda$-presentable objects such that every object is a $\Lambda$-directed colimit of objects from $Z$.
The category $\sB$ is {\em accessible}, if it is {\em $\Lambda$-accessible} for some regular cardinal $\Lambda$.
The following facts are useful:
\begin{enumerate} 
\item $\sB$ is locally presentable if and only if $\sB$ is accessible and cocomplete [by definition].  
\item $\sB$ is locally presentable if and only if $\sB$ is accessible and complete \cite[Cor. 2.47]{AdRo}.
\item If $\sB$ is accessible, then each $X\in\sB$ is presentable \cite[Cor. 2.3.12]{MaPa}. 
\end{enumerate}  

Finally, a functor $H:\sA\rightarrow \sB$ is $\Lambda$-accessible if
both categories $\sA$ and $\sB$ are $\Lambda$-accessible and $H$ preserves $\Lambda$-directed colimits.
The functor $H$ is accessible if it is $\Lambda$-accessible for some regular cardinal $\Lambda$.
Since $\sSets$ is locally presentable, the functor $\sB(X,-)$ is accessible for any object $X$ of an accessible category $\sB$.

\subsection{Categories of comodules and contramodules}
\label{s1.2b}
By {\em comodules} we understand objects in the category of $T$-comodules $\sB_T$.
By {\em contramodules} we understand objects in the category $F$-modules $\sB^F$.
\begin{lemma}
  \label{complete_cocomplete}
  Let $(T \dashv F)$ be an adjoint comonad-monad pair on a complete cocomplete category $\sB$. Then $\sB_T$ is cocomplete and $\sB^F$ is complete.
\end{lemma}
\begin{proof}
  Since $F$ is a monad on $\sB$, the forgetful functor $G^F: \sB^F \to \sB$ creates limits \cite[Th. 3.4.2]{TTT}. Hence, as $\sB$ is complete, so is $\sB^F$. Similarly, since $T$ is a comonad, the forgetful functor $G_T: \sB_T \to \sB$ creates colimits. Hence, $\sB_T$ is cocomplete \cite{HeShi}.
\end{proof}

The questions of cocompleteness of $\sB^F$ and completeness of $\sB_T$ require additional assumptions.
\begin{prop}
  \label{chief_presentable}
Let $(T \dashv F)$  be an enriched adjoint strong comonad-monad pair on a locally presentable, complete, cocomplete left closed monoidal category $\sC$.
Then the categories $\sC_T$ and $\sC^F$ are complete, cocomplete and locally presentable.
\end{prop}
\begin{proof} By Lemma~\ref{complete_cocomplete}, $\sC_T$ is cocomplete and $\sC^F$ is complete.

  The chief $\chf$ is presentable, hence $F$ is accessible (Section~\ref{s1.2a}). 
  The accessibility of $F$ implies that $\sC^F$ is accessible  \cite[Th. 2.78]{AdRo}.
Hence, $\sC^F$ is cocomplete and locally presentable (Section~\ref{s1.2a}). 

The functor $T$ is cocontinuous since it is left adjoint.
Hence,  $T$ is accessible.
By \cite[Cor. 2.8]{HLFV}, $\sC_T$ is accessible.
A cocomplete accessible category is complete and locally presentable (Section~\ref{s1.2a}). 
\end{proof}

\subsection{Comodules and contramodules as enriched categories}
\label{s1.5}
%
The categories of comodules and contramodules admit enrichments in $\sC$ 
that interact with the cofree comodule functor $T^{\sharp}$
and the free contramodule functor $F^{\sharp}$ \cite{Str_1,Str_W}
\[
T^{\sharp}: \sC \rightarrow \sC_T, \ T^{\sharp}(X)=T(X), \ \
F^{\sharp}: \sC \rightarrow \sC^F, \ F^{\sharp}(X)=F(X)
\]
where the comonad structure of $T$ gives the structure map $T(X) \rightarrow TT(X)$ and ditto for $F$.

\begin{prop} \cite[Th. 15]{Str_1}
\label{C_TF enriched}
Suppose a closed symmetric monoidal category $\sC$ satisfies the weak version of completeness
in Assumption~\ref{assum_1}.
\begin{enumerate}
\item If $T$ is a strong comonad on $\sC$, then the comodule category $\sC_T$ admits an enrichment
  such that
  \[(G_T \dashv T^{\sharp}), \ \ \  G_T : \sC_T \rightleftarrows \sC : T^{\sharp}\]
  is an enriched adjunction where $G_T$ is the forgetful functor.
\item If $F$ is a strong monad on $\sC$, then the contramodule category $\sC^F$ admits an enrichment such that
  \[(F^{\sharp} \dashv G^F), \ \ \ F^{\sharp} : \sC \rightleftarrows \sC^F : G^F\] is an enriched adjoint pair 
  where $G^F$ is the forgetful functor.
\end{enumerate}
\end{prop}

We denote {\em the comodule maps object} between objects $(X,\rho_X), (Y,\rho_Y) \in \sC_T$ by $[X,Y]_T$.
It is defined in $\sC$, completely parallel to what one does to define
homomorphisms of comodules or contramodules in the category of vector spaces (cf. Section~\ref{s1.3}).
Let us consider the following two morphisms. The first morphism is
the internal analogue of the composition with $\rho_Y$:
\begin{equation}\label{phiT}
\phi^T_{X,Y} : [X,Y] \xrightarrow{[\Id_X, \rho_Y]} [X,TY].
\end{equation}
The second morphism comes from the enrichment of $T$
\begin{equation}\label{psiT}
\psi^T_{X,Y} : [X,Y] \longrightarrow [TX,TY] \xrightarrow{[\rho_X, \Id_{TY}]} [X,TY].
\end{equation}
  The comodule maps object $[X,Y]_T$ is  the equaliser of $\phi^T_{X,Y}$ and $\psi^T_{X,Y}$.

Similarly, {\em the contramodule maps object} 
between objects $(X,\theta_X), (Y,\theta_Y) \in \sC^F$ is denoted
$[X,Y]^F$.
Again consider the two morphisms
\begin{equation}\label{phiF}
\phi^F_{X,Y} : [X,Y] \longrightarrow [FX,FY] \xrightarrow{[\Id_{FX}, \theta_Y]} [FX,Y], 
\end{equation}
\begin{equation}\label{psiF}
\psi^F_{X,Y} : [X,Y] \xrightarrow{[\theta_X, \Id_Y]} [FX,Y].
\end{equation}
          The contramodule maps object 
    $[X,Y]^F$
    is the equaliser of the maps $\phi^F_{X,Y}$ and $\psi^F_{X,Y}$.
    Notice that both pairs of morphisms admit a common left inverse, coming from the counit of $\chf$.
This is the reason behind Assumption~\ref{assum_1}.

\subsection{Comodule-contramodule correspondence}
\label{s1.6a}
If $(T \dashv F)$ is an adjoint monad-comonad pair, then the categories
$\sC^T$ and $\sC_F$ are isomorphic.
In the case of a comonad-monad pair, the relation between
$\sC_T$ and $\sC^F$
is known as
the comodule-contramodule correspondence.
Notice that the co-contra correspondence exists also in situations not covered by the present set-up, for instance, comodules and contramodules over corings or semi-algebras \cite{Pos2}.
We expect that our methods could be extended to cover such, more general situations.

Let us state the main results of Chapter~\ref{s1}. Their proofs can be found in Sections~\ref{s1.6a_R}, \ref{s1.6a_L}.
and \ref{s1.6a_LR} correspondingly. 
 
\begin{prop} 
  \label{thm_R}
Let $\sC$ be a closed symmetric monoidal category that  satisfies the weak version of completeness
in Assumption~\ref{assum_1}, 
$\chf$ -- a comonoid in $\sC$.
Consider the enriched endofunctors $T=\, - \, \otimes \chf$ and $F=[\chf, - ]$,
together with 
their enriched adjunction $(T \dashv F)$.
\begin{enumerate}
  \item If $X\in \sC_T$, then the hom-object $[\chf,X]_T$ admits a  contramodule structure.
  \item The assignment  $X \mapsto [\chf,X]_T$
determines an enriched functor $R : \sC_T\rightarrow \sC^F$.
\end{enumerate}
\end{prop}

In the case when $\sC$ is the category of vector spaces the functor $R$ admits a left adjoint functor $L$,
given by the contratensor product $Y \mapsto \chf \odot_{\chf} Y$ (cf. Example \ref{s1.3}).
Pushing it through in general categories requires coequalisers as well as equalisers.
Given $(Y,\theta_Y)\in \sC^F$,
consider the following morphisms in $\sC$:
\begin{equation}\label{alpha_beta}
\delta_{Y} : TFY \xrightarrow{T \theta_Y} TY, \ \ 
\beta_{Y} : TFY \xrightarrow{\delta FY} TTFY \xrightarrow{T \epsilon \Id_Y} TY. 
\end{equation}

\begin{prop} 
  \label{thm_L}
  Under the assumptions of Proposition~\ref{thm_R}, suppose further that  $\sC$
  satisfies the weak version of cocompleteness
  in Assumption~\ref{assum_2}.
The assignment of the coequaliser of $\delta_Y$ and $\beta_Y$ to any contramodule $Y$ 
determines an enriched functor $L : \sC^F\rightarrow \sC_T$. 
\end{prop}


\begin{thm}
  \label{adjoint_LR}
  Under the assumptions of Proposition~\ref{thm_L},
$(L \dashv R)$ is a $\sC$-enriched adjoint pair.
\end{thm}

\subsection{Connection with Kleisli categories}
\label{s1.Kl}
Let $\widetilde{\sC_T}$ and $\widetilde{\sC^F}$ be the Kleisli categories.
These are full subcategories of ${\sC_T}$ and ${\sC^F}$ spanned by
the cofree comodules $TX$ and
the free contramodules $FX$ correspondingly. These categories are isomorphic \cite[2.6]{Bohm}. The isomorphisms are given by
\begin{equation} \label{CCC_Kleisli}
\widetilde{\sC_T} \longleftrightarrow \widetilde{\sC^F}, \ \ \ 
TX\longleftrightarrow FX \, . 
\end{equation}
Notice that $R (TX) \cong FX$
but $L(FX)$ appears to be different from $TX$.

The following question is interesting.
A referee has sketched an approach to it via Kan extensions.

\begin{question}
  What is the relation between the functors $L$ and $R$ and
the category isomorphisms 
\eqref{CCC_Kleisli}?
\end{question}

\section{Examples}
\label{s2e}

Now we examine concrete examples of comodules and contramodules.

\subsection{Complexes of vector spaces}
\label{s1.3}
Let $\sC=\KCh$ be the category of chain complexes over a field $\K$.
The tensor product $X\otimes Y$ is just the tensor product of vector spaces.
The internal hom is
\[ [X,Y] = \bigoplus_{d=-\infty}^{\infty} [X,Y]_d \ \mbox{ where } \ [X,Y]_d \coloneqq \prod_i \hom_{\K} (X_i, Y_{i+d}) \, .\] 
Both are chain complexes.
The unit object $\unit$ is the complex $\K[0]$, concentrated in degree zero. 
The zero degree cycles
yield both the global sections and the hom sets:
$$
\Gamma (X)= Z_0 (X) \qquad \mbox{ and } \qquad \sC(X,Y)= Z_0 ([X,Y]).
$$

A comonoid $C$ in $\sC$ is just a DG-coalgebra $C$. Then 
$$T(X) = X \otimes C, \ \ 
F(X)= [C,X].$$ 
If $(X,\rho_X)$ is a $C$-comodule, we write its structure map in Sweedler's $\Sigma$-notation
$$\rho_X (\vx) = \sum_{(\vx)} \vx_{(0)} \otimes \vx_{(1)} \, ,$$
so that the two maps \eqref{phiT} and  \eqref{psiT} are 
$$
\phi^T_{X,Y} (f) (\vx) = \sum_{(f(\vx))} f(\vx)_{(0)} \otimes f(\vx)_{(1)}, \ \ 
\psi^T_{X,Y} (f) (\vx) = \sum_{(\vx)} f(\vx_{(0)}) \otimes \vx_{(1)}.
$$
It follows that the category $\sC_T$ (as defined in Section~\ref{s1.5}) is isomorphic as an ordinary category to the category of DG-comodules over $C$. 

Consider $C$-contramodules  $(X,\theta_X)$ and $(Y,\theta_Y)$, also known as DG-contramodules over $C$.  
Let us inspect the square
\begin{equation}
  \label{contramodule_square}
\begin{CD}
[C, X]  @>{\theta_X}>> X\\
@VV{f\circ}V @VV{f}V\\
[C, Y] @>{\theta_Y}>> Y
\end{CD}
\end{equation}
that depends on a linear map $f:X \rightarrow Y$.
The left-bottom path of the square is $\phi^F_{X,Y}(f)$ and
the top-right path of the square is $\psi^F_{X,Y}(f)$.
Thus, $[X,Y]^F$ is a complex that consists of those $f\in [X,Y]$
that make \eqref{contramodule_square} commutative. 

Now a linear map $f:X\rightarrow Y$ is a homomorphism of DG-contramodules
if $f$ is a map of complexes (degree zero, commutes with differential)
such that \eqref{contramodule_square} is commutative.
In other words, the homomorphisms are the elements of $Z_0 ([X,Y]^F)$.
By definition, $\sC^F(X,Y)=Z_0 ([X,Y]^F)$.
It follows that $\sC^F$ is isomorphic as an ordinary category to the category of DG-contramodules over $C$. 

The adjoint functors $L$ and $R$ are described by Positselski in this case \cite{Pos2}.
They define an equivalence between the coderived category of $C$-comodules and the contraderived category of $C$-contramodules.

\subsection{Specific coalgebra}
\label{s1.3a}
Let us consider the polynomial coalgebra
$C=\K [z]$, $\Delta (z) = 1\otimes z + z \otimes 1$ as a DG-coalgebra with zero differentials.
Let the degree of $z$ be $d\in\Z$.
A $C$-comodule is a chain complex $V$ with a countable family of chain complex maps $\rho_n :V \rightarrow V[nd], \; n \in \N$ ($V[n]$ is the degree shift of $V$) such that
    \begin{equation} 
    \rho : V \rightarrow V \otimes C, \ \
    \rho (\vv) = \sum_n  \rho_n (\vv) \otimes z^n.
    \end{equation}
    It needs to satisfy the unitality and the associativity conditions 
    \begin{equation} \label{unit_assoc}
    \rho_0 (\vv) =\vv, \ \ 
    \rho_m (\rho_n (\vv)) = \binom{m+n}{n}\rho_{m+n} (\vv),
    \end{equation}
    as well as the finiteness condition
    \begin{equation}
      \label{co_fin}
    \forall \vv\in V \; \exists N \; \forall n>N \; \rho_n (\vv) =0.
    \end{equation}
    In characteristic zero \eqref{unit_assoc} is equivalent to 
     \begin{equation} \label{unit_assoc_0}
      \rho_n  =\frac{1}{n!} \rho_1^n \; (\coloneqq \rho_1^{(n)})
       \end{equation}
    so that a $C$-comodule is just a chain complex with a locally nilpotent chain complex operator $\rho_1$.
    
    Similarly, a $C$-contramodule is a chain complex $X$ with a family
    of chain complex maps $\theta_n :X \rightarrow X[nd], \; n \in \N$  such that
    \begin{equation} \label{KX_contramodule}
    \theta : [C,X] \rightarrow X , \ \
    \theta (f ) = \sum_n  \theta_n (f (z^n)).
    \end{equation}
    The unitality and the associativity conditions
    for $\theta$ are \eqref{unit_assoc}, the same as for $\rho$.
    In characteristic zero, it becomes \eqref{unit_assoc_0}, that is, 
    $\theta_n = \theta_1^{(n)}$ for all $n$.
    The finiteness condition is different:
    since $f (z^n)$ can be an arbitrary sequence of elements of $X$, the condition can be stated as  
    \begin{equation}
      \label{contra_fin}
    \forall \; \mbox{sequence} \; (a_n), \; a_n\in X\; \mbox{the sum} \;
    \sum_n \theta_n (a_n) \; \mbox{is well-defined}.
    \end{equation}
    Such well-definedness may or may not result from series convergence in some topology.
    Positselski \cite[0.2]{Pos3} emphasises this point: in general,
    it is an algebraic infinite summation operation.
    It is convenient to think that a $C$-contramodule is equipped with $(U,s)$ where
    $U$ is a subspace of $X[[t]]$
    and
    $s : U \rightarrow X$  is a linear map
    such that  for all  $f \in [C,X]$
    $$
    \sum_n  \theta_n (f (z^n))t^n\in U
    \ \ \mbox{ and } \ \  
    \theta (f ) = s \big(\sum_n  \theta_n (f (z^n))t^n \big) \, .
    $$
    Let $\K$ be of characteristic zero and $d=0$.
    A contramodule of ``topological'' nature is 
    $\K[[x^{-1}]]$ where $\theta_n = \partial_x^{(n)}$. 
    Algebraically, 
    \begin{equation}\label{ex_contra_1}
    U = \{ \sum_n h_nt^n  \,\mid\, h_n \in (x^{-n}) \vartriangleleft \K[[x^{-1}]] \} \, , \ s  (\sum_n h_nt^n) = \sum_n h_n 
    \end{equation}
    and $s$ is well-defined because the calculation of the coefficient in front of each $x^{-n}$ requires only a finite sum.
    

    A contramodule of  ``non-topological'' nature can be constructed similarly to \cite[A.1.1]{Pos2} and \cite[1.5]{Pos3}.
    Let $\widetilde{Y}$ be a $C$-contramodule of sequences
    $$\widetilde{Y}=\{(a_i)\,\mid\, a_i \in \K[[x^{-1}]]\}\, , \ \ \theta_1 = \partial_x \, .$$
    Its summation operation comes from convergence in the $x^{-1}$-adic topology.
    Algebraically, it is given by the formula~\eqref{ex_contra_1} in each position.  
      It has a subcontramodule of quickly convergent sequences $Y$ and the quotient
    $$Y\coloneqq \{(a_i)\in Y \,\mid\,   x^i a_i \to 0 \}\, , \ \ X\coloneqq \widetilde{Y}/Y \, . $$
      The subcontramodule $Y$ is dense in $\widetilde{Y}$.
    Thus, the induced topology on $X$ is antidiscrete and cannot be used to define the summation operation.
    Yet it can be understood algebraically: 
    $$
    U = \{ \sum_n ((a_{n,i})+Y)t^n  \,\mid\, a_{n,i} \in (x^{-n}) \} \, , \ s  (\sum_n ((a_{n,i})+Y) t^n) = (\sum_n a_{n,i}) + Y \, . 
    $$

\subsection{Comonoids and their comodules in the category of sets} \label{comodules_sets}
Let $\sSets$ be the category of sets.  This category has a closed symmetric monoidal structure where the product of sets defines the monoidal structure, and the unit is a one point set $\{p\}$.  In this category the internal hom and the external hom are the same set which we will denote by $[X,Y]$.  Let $\psi = (\psi_1, \psi_2)  : X \to X\times X$ be a comultiplication. The counitality axiom immediately shows that both $\psi_1$ and $\psi_2$ are the identity and so $\psi$ is equal to the diagonal map $\Delta$.  Thus, any set has a unique comonoid structure.  We will fix a base set $C$ and identify this with the comonoid $(C, \Delta, \varepsilon)$ where $\Delta$ is the diagonal map $C \to C \times C$ and $\varepsilon : C \to \{p\}$ is the unique map.  

By definition, a (right) $C$-comodule structure on a set $C$ is a map $\rho: X \to X \times C$ satisfying the usual coassociativity and counitality conditions.  We will use the usual notation $\sSets_C$ for the category of comodules over the monoid $C$.  The counitality immediately shows that there is a unique map $\phi : X \to C$ such that $\rho = 1 \times \phi : X \to X\times C$.

By definition, a set over $C$ is a pair $(X,\phi)$ where $X$ is a set and $\phi : X \to C$ is a function. A morphism of sets over $C$ is a function $f$ such that the following diagram commutes.
$$
\begin{CD}
  X @>{f}>> Y \\
@V{\phi_X}VV @VV{\phi_Y}V\\
C @>>{1}> C \\ 
\end{CD}
$$
We use the notation $(\sSets \!\downarrow\! C)$ for the category of sets over $C$. If $(X,\phi)$ is a set over $C$ then it defines a right $C$-comodule structure on $X$ by setting $\rho = 1 \times \phi : X \rightarrow X \times C$.  

Evidently the correspondence $\rho \longleftrightarrow 1\times \phi$ gives a bijection between the $C$-comodule structures on $X$ and the $C$-set structures on $X$. We state this as a formal proposition.

\begin{prop} \label{chf_sets_comodules}
The above constructions define an isomorphism from the category $\sSets_C$ to the category $(\sSets \!\downarrow\! C)$.
\end{prop}

There is one further point to make.  Let $X = \coprod_{a \in C} X_{a}$ be a disjoint union of a family of sets indexed by $C$.  Then the  set $X$ has a natural map $\phi : X \to C$ defined by
$$\phi(x) = a \ \mbox{ for all } \  x \in X_{a}\, .$$
This, in turn, defines a $C$-comodule structure on $X$.  Then every $C$-comodule is canonically isomorphic to such a disjoint union.

\subsection{Contramodules in the category of sets}
\label{s1.6}
Contramodules over $C$ are a bit more intricate.
By definition, a contramodule over a set $C$ is a set $X$ equipped with a function $\theta : [C, X] \to X$ satisfying the usual contraassociativity and contraunitality conditions.
Now $[C,X] = \prod_{a \in C}X$ and we will sometimes identify a function $\beta : C \to X$ with a list $(\beta (a))_{a \in C}$ of elements in $X$. We can think of $\theta(\beta)$ as the $\theta$-product of the (probably infinite)
list of elements $(\beta(a))_{a \in C}$ in $X$.

Contraunitality tells us that if $f : C \to X$ is the  constant function with value $\vx$, then $\theta(f) = \vx$.  The contraassociativity condition can be rephrased as follows.
Let $\gamma : C \times C \to X$ be a function.
We can think of $\gamma$ as a $C \times C$ matrix with entries in $X$.
The row of $\gamma$ labelled by a fixed element $b \in C$ is the function
$$r_b (\gamma) : C \to X, \ \mbox{ defined by } \ r_{b}(\gamma)(a) = \gamma(b, a).$$
Now we define the {\em row function} of $\gamma$ by
$$
\rho_\gamma : C \to X, \ \ \rho_{\gamma}(a) = \theta(r_{a}(\gamma)).
$$
In other words, $\rho_{\gamma}(a) \in X$ is the $\theta$-product of the elements in the row of the matrix $\gamma$ labelled by $a$.    
We also require the {\em diagonal function} of $\gamma$:
$$\delta_{\gamma} : C \to X, \ \ \delta_{\gamma}(a) = \gamma (a,a).$$
Using these two functions, the contraassociativity condition turns into the equation
\begin{equation} \label{contra_ass_sets}
\theta(\rho_{\gamma}) = \theta(\delta_{\gamma}) \, . 
\end{equation}
We often refer to this equation as the {\em row-diagonal identity}. 

For example, let $C$ be the set $\{1,2\}$.  We identify $\theta$ with a function $\theta : X \times X \to  X$ and write the function $\gamma : C \times C \to X$ in the usual matrix notation
$$
\gamma = \begin{pmatrix}
\vx_{11} & \vx_{12} \\
\vx_{21} & \vx_{22} \\
\end{pmatrix}
$$
In this case, the row-diagonal identity is
\begin{equation}
\theta(\theta(\vx_{11}, \vx_{12}), \theta(\vx_{21}, \vx_{22})) = \theta(\vx_{11}, \vx_{22}).
\end{equation}

\subsection{Product contramodules}
  \label{RX}
Let $Y$ be set over $C$ with a surjective structure map $\phi : Y \to C$.  The set $X = [C,Y]_C$ of sections of $p$ is a non-empty contramodule.   Note that 
$$
Y = \coprod_{a \in C} Y_{a}, \quad X = [C,Y]_C  = \prod_{a \in C} Y_{a}
$$
where $Y_{a} \coloneqq  \phi^{-1}(a) \subseteq Y$. In particular, any product  indexed by $C$ is a contramodule over $C$.  The contramodule structure map $\theta : [C, [C,Y]_C] \to [C,Y]_C$ can be described as follows.
A function $\beta: C \to [C,Y]_C$ is a list 
$\beta =(\beta_{a})_{a \in C}$ of sections of $\phi : X \to C$, i.e.,  $\beta_a\in [C,Y]_C$ for all $a\in C$.  Then $\theta(\beta) \in [C,Y]_C$ is the function
$$
\theta(\beta)(a) = \beta_a (a). 
$$

For example, take $C = \{1,2\}$.  Then the product $Y = Y_1\times Y_2$ equipped with the binary operation
$$
\theta((\vy_1, \vy_2), (\vz_1,\vz_2)) = (\vy_1, \vz_2)
$$
is a contramodule over $\{1,2\}$.  

\subsection{Every contramodule is isomorphic to a product contramodule}
We divide this argument into two steps.  The first is the special case of a contramodule over a set with two elements.  The second is the general case as an adaptation of the special case.

\subsubsection{Contramodules over a set with two elements} \label{s1.6_a}  
Let $X$ be a contramodule over the set $C = \{1,2\}$ with structure map $\theta : X \times X \to X$.  Fix $\vu \in X$ and define $\pi_1, \pi_2 : X \to X$ by
$$
\pi_1(\vx) = \theta(\vx,\vu), \quad \pi_2(\vx) = \theta(\vu,\vx).
$$
Now set
$$
X_1\coloneqq \im(\pi_1) \subseteq X \, , \quad X_2 \coloneqq \im(\pi_2) \subseteq X.
$$
Let us first establish some elementary formulas.  
\begin{lemma} \label{identities}
  The following formulas hold for all $\vx,\vy_1,\vy_2\in X$ and $a,b\in \{1,2\}$.
\begin{enumerate}
\item 
$
  \pi_{a} (\pi_{b}(\vx)) =
  \begin{cases}
  \pi_{a}(\vx)  & \mbox{ if } a=b, \\
\vu & \mbox{ if } a\neq b .
\end{cases}
$
\item
$
\pi_{a}(\theta(\vx_1,\vx_2)) = \pi_{a}(\vx_{a})
$.
\item
If, furthermore,  $(\vx_1, \vx_2) \in X_1 \times X_2$, then
$
\pi_{a}(\theta(\vx_1,\vx_2)) = \vx_{a}
$.
\end{enumerate}
\end{lemma}

\begin{proof}
The row-diagonal identity, combined with the unitality condition applied to the matrix
$$
\begin{pmatrix}
\vx & \vu \\
\vu & \vu \\
\end{pmatrix}
$$
gives the formula $\pi_1\pi_1(\vx) = \pi_1(\vx)$. The same argument using the matrix
$$
\begin{pmatrix}
\vu & \vu\\
\vx & \vu \\
\end{pmatrix}
$$
gives the formula $\pi_2\pi_1(\vx) = u$.  The other formulas follow in a similar fashion.  This proves (1). 

The proof of (2)  when $a =1$  (or $a=2$) is a similar argument using the matrix
$$
\begin{pmatrix}
\vx_1 & \vx_2 \\
\vu & \vu\\
\end{pmatrix}
\qquad \mbox{(correspondingly } \ 
\begin{pmatrix}
\vu & \vu\\
\vx_1 & \vx_2 \\
\end{pmatrix}
\ \mbox{).}
$$

Finally, (3) follows directly from (1) and (2).
\end{proof}

Let us define
$$
\pi = (\pi_1, \pi_2) : X \to X_1 \times X_2
$$
to be the map with components $\pi_1$, $\pi_2$ and
$$
\sigma = \theta\mid_{X_1 \times X_2} : X_1 \times X_2 \to X
$$
to be the restriction of $\theta : X \times X \to X$ to $X_1 \times X_2 \subseteq X \times X$.

\begin{lemma}
The maps $\pi$ and $\sigma$ are inverse isomorphisms of contramodules. 
\end{lemma}
\begin{proof}
Observe that for each $\vx\in X$
\item
$$
\sigma\pi(\vx)  = \theta(\theta(\vx,\vu), \theta(\vu,\vx)) \, . 
$$
The row-diagonal identity applied to the matrix
$$
\begin{pmatrix}
\vx & \vu  \\
\vu & \vx  \\
 \end{pmatrix}
$$
shows that $\theta(\theta(\vx,\vu), \theta(\vu,\vx)) = \theta(\vx,\vx)$ and the unitality condition yields that $\theta(\vx,\vx) = \vx$.  Therefore, $\sigma\pi$ is the identity.

Next for $\vx_1 \in X_1$ and $\vx_2 \in X_2$, Lemma~\ref{identities} ensures that
$$
\pi_1\sigma(\vx_1,\vx_2) =\vx_1, \quad \pi_2\sigma(\vx_1,\vx_2) =\vx_2.
$$
Therefore, $\pi\sigma$ is also the identity. 

Finally,  a simple argument using part~(3) of Lemma~\ref{identities} and the formula for the structure map of $X_1 \times X_2$ shows that $\pi$ is a map of contramodules.
\end{proof} 

\subsubsection{The general case}
The argument in the general case is exactly the same as in the case where $C$ has two elements except that we must replace $2 \times 2$ matrices by the appropriate $C \times C$ matrices.  So let $X$ be a contramodule over $C$.  Fix a point $\vu \in X$.   For each $a \in C$ and $\vx \in X$ define
$$
\delta_{a, \vx} : C   \to X
\ \mbox{ by } \ 
\delta_{a, \vx}(b) =   
  \begin{cases}
  \vx  & \mbox{ if } a=b, \\
\vu & \mbox{ if } a\neq b .
  \end{cases}
$$
Now define $\pi_{a} : X \to X$ by
$$ 
\pi_{a}(\vx) = \theta(\delta_{a, \vx})
$$
and set $X_{a} \coloneqq \im(\pi_{a}) \subseteq X$.
The following elementary formulas is a version of Lemma~\ref{identities} for the general case.  
\begin{lemma}\label{identities_gen} 
The following formulas hold for all $\vx\in X$, $\beta\in [C,X]$ and $a,b\in C$.
\begin{enumerate}
\item 
$  \pi_{a} (\pi_{b}(\vx)) =
  \begin{cases}
  \pi_{a}(\vx)  & \mbox{ if } a=b, \\
\vu & \mbox{ if } a\neq b .
\end{cases}
$
\item
$
\pi_{a}(\theta(\beta)) = \pi_{a}(\beta (a))
$.
\item
If, furthermore,  $\beta (a) \in X_a$ for all $a\in C$, then
$
\pi_{a}(\theta(\beta)) = \beta (a)
$.
\end{enumerate}
\end{lemma}

\begin{proof}
  The proofs follow by writing down the $C\times C$ matrices which are the obvious analogues of the $2\times 2$ matrices in Lemma~\ref{identities}.
  We will write down these general matrices and leave the argument using the unitality and the row-diagonal identities to the reader. 
\begin{enumerate}
\item
  To compute $\pi_{a} (\pi_{a}(\vx))$ we use the matrix $(\vz_{c,d})$ with $\vz_{a,a} = \vx$ and all other entries equal to $\vu$.
  To compute $\pi_{a} (\pi_{b}(\vx))$ we use the matrix $(\vz_{c,d})$ with $\vz_{a,b} = \vx$ and all other entries equal to $\vu$.
\item
To prove (2) we use the $C \times C$ matrix $(\vz_{c,d})$ with all entries $\vu$ except in the row labelled $a$.  In this row $\vz_{a,b} = \beta (b)$.
\end{enumerate}
\item
Finally, to prove (3), note that since $\beta (a) \in X_{a}$ it follows that $ \beta(a) = \pi_{a}(\vz)$ for some $\vz \in X$.  The formula follows from (1) and (2).
\end{proof}

As above we write 
$$
\pi = (\pi_a):  X \to \prod_{a \in C} X_{a}
$$
for the map with components $\pi_{a}$ and
$$
\sigma = \theta\mid_{\prod_{a \in C} X_{a}}: \prod_{a \in C} X_{a} \to X
$$
for the restriction of the contramodule structure map $\theta :  \prod_{\alpha \in C} X \to X$.  In terms of functions, this subset of $[C,X] = \prod_{a \in C}X$ corresponds to the set of such functions $\beta : C \to X$
that $\beta (a) \in X_{a}$. 

\begin{thm}
\label{any_contraset}
The maps $\pi$ and $\sigma$ are inverse isomorphisms of contramodules.
\end{thm}

\begin{proof}
It follows immediately from the definitions of $\pi$ and $\sigma$ that
$$
\sigma(\pi(\vx)) = \theta(\theta ((\delta_{a, \vx})_{a \in C})) \, . 
$$
To compute this by the row-diagonal identity, we introduce the $C\times C$ matrix $\gamma=(\vz_{ab})$ defined by
$$
\vz_{aa}= \vx, \quad \vz_{ab} = \vu \quad\text{if $a \neq b$.}
$$
The row labelled by $a$ of $\gamma$ is $r_a(\gamma) = \delta_{a, \vx}$ and the corresponding row function 
is precisely
$$
\rho_\gamma : C\to X, \ 
\rho_\gamma (a) = \theta(\delta_{a, \vx}) = \pi_{a}(\vx).
$$
The diagonal entries of $\gamma$ are all equal to $\vx$ and, by the unitality condition, it follows that $\theta(\delta_{\gamma}) = \vx$.
Now the row-diagonal identity inplies that $\sigma\pi$ is the identity: 
$$
\vx = \theta(\delta_{\gamma}) = \theta(\rho_{\gamma}) = \theta(\theta ((\delta_{a, \vx})_{a \in C})) = \sigma(\pi(x)).
$$

The facts that $\pi$ is a map of contramodules and that $\pi\sigma$ is the identity follow directly from Lemma~\ref{identities_gen}. 
\end{proof}

\subsection{The co-contra correspondence in the category of sets}
\label{s1.6_b}
We now consider the functors $R : \sSets_C \to \sSets^C$ and $L : \sSets^C \to \sSets_C$. If $X$ is a comodule over $C$, $R(X) = [C,X]_C$ is the set of sections of the structure map $\phi : X \to C$.
We will say that a comodule over $C$ is {\em degenerate} if the structure map $\phi$ is not surjective.
  Notice that $R(X)$ is the empty set, if $X$ is degenerate.
By Theorem~\ref{any_contraset}, every non-empty contramodule over $C$ is isomorphic to $R(X)$ for some $C$-comodule $X$.  

The functor $L$ is more intricate.  Given $Y$, a contramodule over $C$, we have two maps
$$
\eta,\nu \ : [C,Y] \times C \to Y \times C
$$
$$
\eta( \beta, a) = (\beta(a), a), \quad \nu(\beta,a) = (\theta(\beta), a).
$$
Then $L(Y)$ is the coequaliser of these two maps.  In this section we prove the following theorem.  

\begin{thm}  \label{set_CCC}
Let $C$ be a set. The functors $L$ and $R$  are quasi-inverse equivalences between the category of non-degenerate $C$-comodules and the category of non-empty $C$-contramodules.   
\end{thm}

The coequaliser in the definition of $L(Y)$ is the the quotient of $Y \times C$ by an equivalence relation.  The key to the proof of the theorem is to understand this equivalence relation in the case where $Y = R(X) = [C,X]_C$.
Let $\sim$ be the equivalence relation on $[C,X]_C\times C$ defined as follows:  
\begin{equation} \label{relaion_sim}
(\beta,a) \sim (\gamma,b) \Longleftrightarrow a = b \ \mbox{ and } \ \beta(a) = \gamma(b).
\end{equation}

\begin{lemma} \label{LR_description}
Let $X$ be a comodule over $C$.  Then
$$
L(R(X) ) = ([C,X]_C\times C)/\sim.
$$
\end{lemma}

\begin{proof}
  From the definition of the coequaliser, $L(R(X))$ is the quotient of \newline
  $[C,X]_C  \times C$ by the equivalence relation generated by the binary relation $\approx$.  The relation $\approx$ is defined by one of the following equivalent three statements:
\begin{enumerate} 
\item $(\beta,a)\approx (\gamma,b)$,
\item there exist $\psi : C \to [C,X]_C$ and $c \in C$ such that $\eta(\psi,c) = (\beta,a)$ and $\nu(\psi,c) = (\gamma, b)$,
\item $a = b$ and there is a function $\psi: C \to [C,X]_C$ such that for all $c \in C$, $\psi(a)(c) = \beta(c)$ and $\psi(c)(c) = \gamma(c)$. 
\end{enumerate}
The lemma immediately follows from the equivalence
\begin{equation} \label{2rel_eq}
(\beta,a)\approx(\gamma,b) \Longleftrightarrow (\beta,a) \sim (\gamma,b)
\end{equation}
that we are going to establish in the rest of this proof. 
The statement~(3) from the list above tells us that
$$
(\beta,a)\approx(\gamma,b) \ \implies \ \beta(a) = \psi(a)(a) = \gamma(a).
$$
This proves the direct implication in~\eqref{2rel_eq}. 
To prove the reverse implication, pick $\beta,\gamma \in [C,X]_C$ such that $\beta(a) = \gamma(a)$.  Define $\psi : C \to [C,X]_C$ by
$$
\psi (a) = \beta, \quad \psi (b)=  \gamma  \quad\text{if $a \neq b$}.
$$
Since $\beta,\gamma \in [C,X]_C$ it is clear that $\psi (c) \in [C,X]_C$ for all $c \in C$.  The statement~(3) from the list implies that
$(\beta,a)\approx(\gamma,a)$. This completes the proof.
\end{proof}

Now we prove  Theorem~\ref{set_CCC}. 
\begin{proof}  
Let $X$ be a non-degenerate $C$-comodule.  Consider the map
$$
\varpi_X : [C,X]_C\times C \to X, \quad \varpi_X (\beta,a) = \beta(a).
$$
This map is surjective. 
Therefore, using the relation~\eqref{relaion_sim}, we conclude that the quotient map 
$$
\overline{\varpi}_X : \left([C,X]_C\times C\right)/_{\sim} \longrightarrow  X
$$
is an isomorphism.  By Lemma~\ref{LR_description}, 
$$
L(R(X)) = \left([C,X]_C\times C\right)/\sim
$$
and so we get a natural isomorphism
$$
L(R(X)) \to X
$$
It is not difficult to check that this natural isomorphism is the counit of the adjunction $(L\dashv R)$.

Now let $Y$ be a $C$-contramodule.  We have a natural transformation $Y\to R(L(Y))$, the unit of the adjunction.
Choose a $C$-comodule $X$ and an isomorphism $\varphi : R(X) \to Y$.  This gives a commutative diagram
$$
\begin{CD}
R(X) @>>> R(L(RX))   \\
@VV{\varphi}V @VV{RL(\varphi)}V \\
Y @>>> R(L(Y)) 
\end{CD}
$$
The top horizontal arrow is an isomorphism as are the two vertical arrows.  This proves that the map $Y\to R(L(Y))$ is also an isomorphism.
\end{proof}

\subsection{Simplicial sets}
\label{SSets_1}
Let $\sS$ be the category of simplicial sets. This is a cartesian category, that is, the monoidal product is the categorical product.
It is a closed symmetric monoidal category:
$$
(X\times Y)_n = X_n \times Y_n, \ \ \
[Y,X]_n = \sS (Y \times \Delta[n] ,X) 
$$ 
at each level $n$, where $\Delta[n]\in\sS$ is the standard $n$-simplex.
As in the start of Section~\ref{s1.6}, a comonoid in $\sC$ is a simplicial set $C=(C_n)$ with the diagonal map $C \rightarrow C \times C$.

Similarly to  \eqref{slice_isomorphic} and Proposition~\ref{chf_sets_comodules}, $\sS_T$ is isomorphic to the overcategory 
$(\sS \!\downarrow\! C)$ (c.f. \cite{Hir2}). Thus, a $C$-comodule $M=(M_n)$ is a simplicial set with a $C_n$-set structure $\phi_n : M_n\rightarrow C_n$ at each level $n$. The compatibility condition is commutation of $\phi$ with the simplicial set structure maps:
$$ \phi_n \circ M(f) = C(f) \circ \phi_m
$$
for all non-decreasing functions $f:[n] \rightarrow [m]$.
Here by $[n]$ we denote the ordered set $\{0,\ldots, n\}$. 

Let us briefly examine a $C$-contramodule $(X=(X_n),\theta)$. Its structure map
$\theta =(\theta_n) \in \sS ([C,X], X)$
consists of functions for each $n$
$$
\theta_n : [C,X]_n = \sS (C \times \Delta[n] ,X) \rightarrow X_n
$$
satisfying the contraassociativity and contraunitality conditions.
%

\section{Deferred Proofs}
\label{s_appendix}

\subsection{Enriched and ordinary (co)equalisers}
\label{s1.4ax}
Let $\sC$ be a closed symmetric monoidal category. 
Both $\sC$ and $\sC^{op}$ are enriched in $\sC$:
\[ [X,Y]_{\sC} =  [X,Y]\, , \ \ \ [X,Y]_{\sC^{op}} =  [Y,X] \, . \]
The equaliser of a pair $f,g:X\rightrightarrows Y$
represents a functor
$$
\sC^{op}\rightarrow \sSets, \ \
Z \mapsto \eq ( f\circ ,g\circ :\sC(Z,X)\rightrightarrows \sC(Z,Y) ).
$$
Similarly, {\em the enriched equaliser} of this pair is a map $h:K\rightarrow X$
such that the functor 
$$
E: \sC^{op}\rightarrow \sC, \ \
Z \mapsto \eq ( \widetilde{f}_Z ,\widetilde{g}_Z : [Z,X]\rightrightarrows [Z,Y] ).
$$
is represented by $K$ with the natural isomorphism $[-,K] \rightarrow E$ given by the evaluation $\widetilde{h}_{-}$ .
Dually, {\em the enriched coequaliser} of the pair  $f,g:X\rightrightarrows Y$
is a map $d:Y\rightarrow K$
such the functor 
$$
F: \sC\rightarrow \sC, \ \
Z \mapsto \eq ( {}_{Z}\widetilde{f} ,{}_{Z}\widetilde{g} : [Y,Z]\rightrightarrows [X,Z] ), 
$$
where ${}_{Z}\widetilde{f}$ and ${}_{Z}\widetilde{g}$ are evaluations on the other side, 
is represented by $K$ with the natural isomorphism $[K,-] \rightarrow F$ is given by the evaluation ${}_{-}\widetilde{d}$.
\begin{lemma}
  An equaliser is an enriched equaliser.
  A coequaliser is an enriched coequaliser. 
\end{lemma}
\begin{proof} Suppose $h:K\rightarrow X$ is the equaliser of a pair $f,g:X\rightrightarrows Y$.
  The functor $[Z,-]$ preserves limits because it is a right adjoint.
  Thus, $[Z,K]$ is the  equaliser of the pair  $\widetilde{f}_Z ,\widetilde{g}_Z : [Z,X]\rightrightarrows [Z,Y]$.
  Hence, $h:K\rightarrow X$ is the enriched equaliser. 
  The proof for coequalisers is similar.
\end{proof}

\subsection{Functor from comodules to contramodules}
\label{s1.6a_R}
We prove Proposition~\ref{thm_R} in this section.

The map $\phi^T_{C,X}: FX \rightarrow FTX$ in \eqref{phiT}
is a homomorphism of free contramodules. So is the map in \eqref{psiT}.
This becomes clear after rewriting it using the adjunctions
\begin{equation} \label{psiTCX}
\psi^T_{C,X} : FX=[C,X] \longrightarrow F^2TX  \longrightarrow FTX.
\end{equation}
Since the forgetful functor $G^F: \sC^F \to \sC$ is left adjoint, it preserves limits.
Thus, $RX$ is the equaliser in $\sC^F$ and a contramodule.

We need to show that $R$ is an enriched functor. Let $\sD (\sC^F)$ be the category
of diagrams
\begin{equation}\label{cat_diag}
  \Psi: \sJ  \longrightarrow \sC^F \, , \ \ \ \sJ\coloneqq (\bullet \rightrightarrows \bullet) \, .
  \end{equation}
By trivially enriching the index category $\sJ$, we make the category $\sD (\sC^F)$ enriched.
Moreover, the equaliser $\eq: \sD (\sC^F)\rightarrow \sC^F$ becomes an enriched functor.

By inspection, the assignment $X\mapsto (\phi^T_{C,X},\psi^T_{C,X})$ is an enriched functor
$R_0: \sC_T \rightarrow \sD (\sC^F)$. The functor $R$ is a composition
of two enriched functors $R_0$ and $\eq$, hence, enriched.

\subsection{Functor from contramodules to comodules}
\label{s1.6a_L}
A proof of Proposition~\ref{thm_L} is similar to the proof of Proposition~\ref{thm_R}.
%
%
The maps 
$\beta_{Y}$
and
$\delta_{Y}$
in the definition of $LY$ (see \eqref{alpha_beta})
are morphisms of cofree comodules.
Their common right inverse is
\[TY \xrightarrow{T \eta_{F}(Y)} TFY \, . \]
By Assumption~\ref{assum_2},
$\beta_{Y}$ and $\delta_{Y}$
admit a coequaliser $LY$. 
Since the forgetful functor $G_T: \sC_T \to \sC$ is right adjoint, it preserves colimits.
Thus, $LY$ is the coequaliser in $\sC_T$ and a comodule.

To show that $L$ is enriched, consider $\sD (\sC_T)$ (cf.~\eqref{cat_diag}).
By trivially enriching the index category $\sJ$, we make the category $\sD (\sC_T)$ enriched.
Moreover, the coequaliser $\coeq: \sD (\sC_T)\rightarrow \sC_T$ becomes an enriched functor.

By inspection, the assignment $X\mapsto (\beta_{Y},\delta_{Y})$ is an enriched functor
$R_L: \sC^F \rightarrow \sD (\sC_T)$. The functor $L$ is a composition
of two enriched functors $L_0$ and $\coeq$, hence, enriched.

\subsection{Enriched adjunction}
\label{s1.6a_LR}
We start with a useful fact.
\begin{lemma} (cf. \cite[1.7 and 1.8]{Kelly}) \label{enriched_naturality}
Let $\sC$ be a closed symmetric monoidal category.  Let $\sA$ be  a $\sC$-enriched category. 
\begin{enumerate}
\item If $\sB$ is another $\sC$-enriched category and $H: \sA \to \sB$ is a $\sC$-enriched functor, then the maps
  \[H_{X,Y}\, : \, \sA (X,Y) \rightarrow \sB(HX,HY) \ \ \mbox{ (cf. \eqref{enriched_functor} )}   \]
  are $\sC$-natural in both $X$ and $Y$.
\item The internal hom $[X,Y]_\sA$ (cf. \eqref{hom_compl})  is $\sC$-natural in both $X$ and $Y$.
\item The enriched composition
  $c^\sA_{X,Y,Z}$
(cf. \eqref{hom_compl})
is $\sC$-natural in $X$ and $Z$, and $\sC$-extranatural in $Y$.
\end{enumerate}
\end{lemma}

We proceed with the proof of the adjunction $(L\dashv R)$ in Theorem~\ref{adjoint_LR}.

\begin{proof} 
  To show that $(L, R)$ is a $\sC$-enriched adjoint pair we need to show that there is a $\sC$-natural isomorphism of bifunctors
$$ [L X, Y]_T \cong [X, R Y]^F.$$

Note that by Lemma~\ref{enriched_naturality} the internal hom bifunctor $[-,-]$ is a $\sC$-natural transformation. Thus, the adjunction $(T\dashv F)$ becomes an isomorphism of $\sC$-enriched bifunctors
$$
[TX,Y] \cong [X,FY] \, . 
$$
Moreover, we have $[TX,Y]_T \cong [X,R Y]$ and $[L X,Y] \cong [X,FY]^F$ as objects in $\sC$. Note that the maps 
\[ \phi^T_{TX,TY}, \psi^T_{TX,TY}: [TX,Y] \rightrightarrows [TX,TY] \]
are adjuncts of the maps
\[ [\Id, \phi^T_{\chf,X}], [\Id, \psi^T_{\chf,X}]: [X, FY] \rightrightarrows [X,FTY] \, .\]
Observe that the functor $[X,-]$ is a right adjoint and thus preserves kernels. Combined with the fact that $[TX,Y] \cong [X,FY]$ is an isomorphism of bifunctors, we can deduce that every map which equalises the pair $(\phi^T_{TX,TY}, \psi^T_{TX,TY})$ also equalises $([\Id, \phi^T_{\chf,X}], [\Id, \psi^T_{\chf,X}])$. This implies the isomorphism $[TX,Y]_T \cong [X,R Y]$. 

Again, by Lemma~\ref{enriched_naturality} this isomorphism is, in fact, a $\sC$-enriched isomorphism of enriched bifunctors. The argument for $[L X,Y] \cong [X,FY]^F$ is analogous.

We can complete the proof by observing that the following squares are cartesian in $\sC$:
$$
\begin{CD}
[TX,Y] @<{[\gamma_X, \Id]}<< [L X, Y] \\
@A{\gamma^T_{TX,Y}}AA @A{\gamma^T_{L X, Y}}AA\\
[TX,Y]_T @<<{\psi}< [L X, Y]_T\\
\end{CD}
\ \  \cong \ \   
\begin{CD}
[X,FY] @<{\gamma^F_{X,FY}}<< [X, FY]^F \\
@AA{[\Id, \gamma^T_{\chf,X}]}A @AA{\phi}A\\
[X, R Y] @<<{\gamma^F_{X,R Y}}< [X, R Y]^F\\
\end{CD}
$$

Let $d_1 \coloneqq [\gamma_X, \Id] \circ \gamma^T_{L X,Y}$ and $d_2 \coloneqq [\Id, \gamma^T_{\chf, X}] \circ \gamma^F_{X,R Y}$. The maps $d_1$ and $d_2$ clearly equalise the pairs $(\phi^T_{TX,TY},\psi^T_{TX,TY})$ and $([\Id, \phi^T_{\chf,X}], [\Id, \psi^T_{\chf,X}])$ respectively. Thus, by definition $d_1=\gamma^T_{TX,Y}\circ \psi$, i.e., the left square commutes. The universal property of the equaliser implies that $[L X,Y]_T$ is a pullback. A similar argument shows that  the square on the right is cartesian. The existence of the $\sC$-enriched isomorphisms of bifunctors explained above completes the proof.



  \end{proof}

\subsection{Change of comonoid}
\label{s1.Coc}
Now we collect standard technical facts on the behaviour of comodules and contramodules under a morphism $f:\chf \rightarrow \widehat{\chf}$ of comonoids in $\sC$. We omit the proofs.  

We denote the two comonad-monad adjoint pairs with chiefs $\chf$ and  $\widehat{\chf}$ by $(T\dashv F)$ and $(\widehat{T}\dashv \widehat{F})$. Clearly, we have {\em restriction functors} 
  $$
\sRes : \sC_T \rightarrow \sC_{\widehat{T}}, \ \ \sRes (M, \rho: M \rightarrow T M) = (M, (f\otimes \Id_M) \circ \rho),
$$
$$
\sRes : \sC^{F} \rightarrow \sC^{\widehat{F}}, \ \ \sRes (N, \theta: FM \rightarrow M) = (M, \theta \circ [f,\Id_M]).
$$

Besides the comodules and the contramodules, we would like to consider the overcategory (or slice category) $(\sC \!\downarrow\! \chf)$.
Again, there is a restriction functor
$$
\sRes : (\sC \!\downarrow\! \chf) \rightarrow (\sC \!\downarrow\! \widehat{\chf}), \ \ 
\sRes (M, \phi: M \rightarrow \chf) = (M, f \circ \phi).
$$
All three functors deserve the same notation because they are essentially the ``same'' functor, at least they are the same on objects. The similarity breaks down when we consider the existence of induction functors, forcing us to use different notations. 

We start with the overcategory because it is the easiest one to understand.
\begin{prop} \label{right_adjoint_slice}
  Let $\chf$, $\widehat{\chf}$ be any objects of $\sC$, $f\in \sC(\chf,\widehat{\chf})$. Then
$$
\sInd\!\!\downarrow : (\sC\!\downarrow\!{\widehat{\chf}}) \rightarrow (\sC\!\downarrow\!{\chf}), \ \ \sInd\!\!\downarrow (P, \phi: P \rightarrow \widehat{\chf}) = (P\times_{\widehat{\chf}}{\chf}, \pi_2),
$$
where $\pi_2$ is the projection onto the second component, is a $\sC$-enriched functor, $\sC$-enriched right adjoint to $\sRes$.
\end{prop}
This proposition is an enriched version of the standard fact \cite[Lemma 7.6.6]{Hir}. 


Our comodules are right comodules since $T= - \otimes \chf$. Similarly, there is a category of left comodules, ${}_{T}\sC$, comodules over the comonad $T^\prime = \chf\otimes -$. The comonoid $\chf$ is naturally an object of both ${}_{T}\sC$ and $\sC_T$. In fact, it is a bicomodule in a suitable sense.
\begin{prop} \label{right_adjoint_T} Suppose that the symmetric monoidal category $\sC$ satisfies Assumption~\ref{assum_1}. 
  \begin{enumerate}
  \item There exists a cotensor product, an enriched in $\sC$ bifunctor
    $$
    - \square_{\chf} - : {\sC}_{T}\times {}_{T}\sC \rightarrow \sC,
    $$
    where  $M \square_{\chf} N$ is the equaliser of the pair of maps
    $$
    \rho_M \otimes \Id_N, \, a^{-1}_{M,\chf,N} \circ (\Id_M \otimes \rho_N) \, :\, M \otimes N \rightrightarrows (M \otimes \chf) \otimes N.
    $$
  \item
If $f$ is a morphism of comonoids and the monad $T=\, - \, \otimes \chf$ preserves equalisers of pairs of morphisms, then
$$
\sInd_T : \sC_{\widehat{T}} \rightarrow \sC_{T}, \ \ \sInd_T (M, \rho: M \rightarrow \widehat{T}(M)) = (M\square_{\widehat{\chf}}{\chf}, \widetilde{\rho})
$$
is  a $\sC$-enriched functor, where the structure morphism $\widetilde{\rho}$ appears in the diagram
\begin{center}
\begin{tikzcd}
M\square_{\widehat{\chf}}\chf \arrow[r] \arrow[d, dashrightarrow, "\widetilde{\rho}"]   & {M\otimes \chf} \arrow{d}[swap]{a^{-1}_{\cdot,\cdot,\cdot}\circ(\Id\otimes\Delta)} \arrow[r, shift left]\arrow[r,shift right]  & {(M\otimes \widehat{\chf})\otimes \chf} \arrow{d}[swap]{a^{-1}_{\cdot,\cdot,\cdot}\circ(\Id\otimes\Delta)}  \\
(M\square_{\widehat{\chf}}\chf \arrow[r])\otimes \chf   \arrow{r}
   & {(M\otimes {\chf})\otimes \chf} \arrow[r, shift left]\arrow[r,shift right]  & {((M\otimes \widehat{\chf})\otimes \chf) \otimes \chf }
\end{tikzcd}
\end{center}
with equalisers in both rows and three commutative squares
(the right square is commutative as soon as only the top or only the bottom arrows are taken).
Furthermore, $(\sRes \dashv \sInd_T)$ is an enriched adjunction.
\end{enumerate}
\end{prop}
For coalgebras over rings this proposition is well known \cite[11.9]{BrWi}. 


If $T$ is continuous, then it preserves the equalisers.
Similarly in Proposition~\ref{left_adjoint_F} below, if $F$ is cocontinuous, then it preserves the coequalisers.
In the category of chain complexes over a commutative ring $\K$ (see Section~\ref{s4e}), these are conditions for $\chf$ to be flat and projective correspondingly.
See also Section~\ref{s1.6b}.


\begin{prop} \label{left_adjoint_F}
Suppose that the symmetric monoidal category $\sC$ satisfies Assumptions \ref{assum_1} and \ref{assum_2}. 
  \begin{enumerate}
  \item There exists a cohom, a $\sC$-enriched bifunctor
    $$
    \Cohom_{\chf} (-,-) : {\sC}_{T}\times \sC^F \rightarrow \sC,
    $$
    where $\Cohom_{\chf} (M, P)$ is the coequaliser of the pair of maps
    $$
    ad_{M,P} \circ [\rho_M,\Id_P], \, [\Id_M, \theta_P] \, :\, [M,F(P)] \rightrightarrows [M,P],
    $$
    where $ad_{M,P}$ is the internal adjunction map.
  \item  If 
    $f$ is a morphism of comonoids and the comonad $F=[\chf, - ]$ is cocontinuous, then
$$
\sCoi^F : \sC^{\widehat{F}} \rightarrow \sC^{F}, \ \ \sCoi^F (P, \theta: \widehat{F}(P)\rightarrow P) =  (\Cohom_{\widehat{\chf}} (\chf , P), \widetilde{\theta})
$$
is  a $\sC$-enriched functor, where the structure morphism $\widetilde{\theta}$ appears in the diagram
\begin{center}
\begin{tikzcd}
  {[\chf, [\widehat{\chf},P]]} \arrow[r, shift left]\arrow[r,shift right]
  & {[\chf,P]}   \arrow[r]
  & \Cohom_{\widehat{\chf}}(\chf,P)   \\
  {[\chf, [{\chf}\otimes \widehat{\chf},P]]} \arrow[r, shift left]\arrow[r,shift right] \arrow{u}[swap]{[\Delta, \Id] \circ a^{-1}_{\cdot,\cdot,\cdot}\circ ad_{\cdot,\cdot}}
  & {[\chf, [{\chf},P]]} \arrow{r} \arrow{u}[swap]{[\Delta, \Id]\circ ad_{\cdot,\cdot}}
  & {[\chf,\Cohom_{\widehat{\chf}}(\chf,P)]} \arrow[u, dashrightarrow, "\widetilde{\theta}"]
\end{tikzcd}
\end{center}
with coequalisers in both rows and three commutative squares
(the leftt square is commutative as soon as only the top or only the bottom arrows are taken).
Furthermore, $(\sCoi^F \dashv \sRes)$ is an enriched adjunction.
\end{enumerate}
\end{prop}
In the context of DG-coalgebras over rings this proposition was discovered by Positselski \cite[2.2]{Pos}. 
We finish this section with a question, reminiscent of the standard cohom-defining property in linear categories: 
\begin{question}
  Suppose that the symmetric monoidal category $\sC$ satisfies Assumptions \ref{assum_1} and \ref{assum_2}.
  Does there exist a $\sC$-enriched natural equivalence of trifunctors $\sC{}_{T}\times\, {}_{T}\sC \times \sC \rightarrow \sC$
$$ [M\square_{\chf} N, X] \cong \Cohom_{\chf} (M, [N,X]) \, ?$$  
\end{question}

\subsection{Induction for contrasets}
\label{s1.6b}
Observe that in the category $\sSets$ the comonad $T$ is continuous for any $\chf$.
Thus, for any function $f:\chf \rightarrow \widehat{\chf}$, we have the induction functor for comodules as in Proposition~\ref{right_adjoint_T}.

This agrees well with the isomorphism of categories in Proposition~\ref{chf_sets_comodules}. Indeed, the induction functor for the overcategories $(\sSets\!\downarrow\!{\chf})$ does not require any additional assumptions (cf.  Proposition~\ref{right_adjoint_slice}).

On the other hand, $F$ is not cocontinuous if $|\chf| \geq 2$.
Let $\chf$ be a 2-element set. In this case,  $F(X) = X^2$ for any set $X$. Look at the coequaliser of two maps from a point
$$
\unit \rightrightarrows X \stackrel{coeq.}{\dashrightarrow} X/\!\sim \, .
$$
Here  $X/\sim$ is obtained from $X$ by identifying the two image  points. Apply $F$:
$$
F(\unit) = \unit \rightrightarrows F(X) \stackrel{coeq.}{\dashrightarrow} (X^2)/\! \; \sim \neq (X/\!\sim)^2 = F(X/\!\sim )\, .
$$
Thus, Proposition~\ref{left_adjoint_F} gives us no coinduction for contramodules in $\sSets$.

Let us discuss restriction. In light of Theorem~\ref{any_contraset},
a contramodule  $(P,\theta_P)  \in \sSets^{\chf}$ is represented as a product $(P,\theta_P)  = \prod_{x\in {\chf}} P_x$. Its restriction has similar representation:
\begin{equation}
  \label{restriction_explicit}
  (\widehat{P} ,\widehat{\theta_P}) = Res (P,\theta_P)= \prod_{z\in \widehat\chf} \widehat{P}_{z} \, ,
  \mbox{ where } \
\widehat{P}_{z} \,   = \prod_{y\in f^{-1}(z)} P_{y} \, .
\end{equation}
Notice that if $z$ is not in the image of $f$, then $\widehat{P}_{z}$ is a 1-element set. Now it is time to address induction.
\begin{prop} \label{right_adjoint_F}
  Let $\chf, \widehat{\chf}\in\sSets$, $f\in \sSets(\chf,\widehat{\chf})$. Then there exists a functor
$$
  \sInd^F : \sSets^{\widehat{\chf}} \rightarrow \sSets^{\chf},
$$
  left adjoint to $\sRes$.
\end{prop}
\begin{proof} A function $f$ is a composition of a surjection $f_1$ and an injection $f_2$:
$$
f: \chf \xrightarrow{f_1} \widetilde{\chf} = \mbox{Im} (f) \xrightarrow{f_2} \widehat{\chf} \, .
$$
It suffices to define a left adjoint functor to $\sRes$ for injections and surjections separately. Then $\sInd$ is a composition of these two functors.

  If $f$ is surjective, we can define the induction functor as a composition
\begin{equation} \label{Ind_surjection}
  \sInd^F : \sSets^{\widehat{\chf}} \xrightarrow{L} \sSets_{\widehat{\chf}} \xrightarrow{\sInd_T} \sSets_{\chf} \xrightarrow{R} \sSets_{\chf} \, .
\end{equation}
In this case a non-degenerate comodule remains non-degenerate after induction. Thus, the non-empty contramodules turn into non-degenerate comodules and vice versa. The empty contramodule $\emptyset$ remains empty, going through these functors.
It follows from Proposition~\ref{right_adjoint_slice} and Theorem~\ref{set_CCC} that this is a left adjoint. 

Now let us assume that $f$ is injective. We can define induction explicitly as
\begin{equation}
  \label{induction_injective}
 (\widetilde{X},\widetilde{\theta_X}) = \sInd^F (X,\theta_X)= \prod_{y\in \chf} X_{f(y)} \, ,
  \mbox{ whence } \ 
  (X,\theta_X)  = \prod_{y\in \widehat{\chf}} X_y \, . 
\end{equation}
To prove that this is a left adjoint, we just need to translate the representation in
Theorem~\ref{any_contraset}
to an explicit calculation of homs:
$$
[\widetilde{X},P]^{\chf}
=
\prod_{z\in\chf} [X_{f(z)}, P_z]
\stackrel{(\Asterisk)}{=}
\prod_{y\in\widehat{\chf}} [X_y, \widehat{P}_y]
=
[X,\widehat{P}]^{\widehat{\chf}}
$$
where the equality $(\Asterisk)$ holds true because $\widehat{P}_y=P_y$ if $y\in\mbox{Im}(f)$ and $\widehat{P}_y$ is a 1-element set otherwise.
\end{proof}

It follows from Proposition~\ref{right_adjoint_F} that equation~\eqref{induction_injective} essentially defines the induced contramodule for a general $f$ as well. If $(X,\theta)  = \prod_{y\in \widehat{\chf}} X_y \in \sSets^{\widehat{\chf}}$, then 
\begin{equation}
  \label{induction_explicit}
 (\widetilde{X},\widetilde{\theta_X}) = \sInd^F (X,\theta_X)= \prod_{z\in \chf} \widetilde{X}_{z} \, ,
  \mbox{ where } \
  \widetilde{X}_{z} = X_{f(z)} \, .
\end{equation}

\section{Model Categories}
\label{s2}

\subsection{Model structures}
\label{s2.1}
Let  $\sB$ be a model category, which we assume to be complete and cocomplete \cite{Hir,Hov}.
The structure classes of morphisms are denoted $\Cof$ for cofibrations, $\WEq$ for weak equivalences and $\Fib$ for fibrations.
Given a morphism $f$, we write its factorisations in the following way: 
$$
f:X\xrightarrow{f^{\prime} \ \Cof} Y\xrightarrow{f^{\prime\prime} \ \Fib\WEq} Z, \ 
f:X\xrightarrow{f^{\prime} \ \Cof\WEq} Y\xrightarrow{f^{\prime\prime} \ \Fib} Z.
$$
Unlike \cite[Def. 1.1.4]{Hov}, we do not automatically assume that the factorisations are endofunctors on the category of maps $\Map (\sB)$ 
(also called the category of squares or the category of arrows). Recall that $\Map (\sB)$ has the maps in $\sB$ as objects and commutative squares in $\sB$ as morphisms.


An object $X\in \sB$ is cofibrant if the map from the initial object ${\emptyset}_X: \emptyset\rightarrow X$ is a cofibration. Similarly, an object $X\in \sB$ is fibrant if the map to the terminal object ${\mathbf 1}_X: X\rightarrow {\mathbf 1}$ is a fibration.
By $X_{\bC}$ and $X_{\bF}$ we denote cofibrant and fibrant replacements of $X$. 
The full subcategory of cofibrant (or fibrant, or cofibrant and fibrant) objects is denoted $\sB_{\bC}$ (or $\sB_{\bF}$, or $\sB_{\bC\bF}$).

A model category $\sB$ is called accessible if $\sB$ is a locally presentable category and both factorisations can be realised by
accessible endofunctors on  $\Map (\sB)$.

\subsection{Model structures on closed monoidal categories}
\label{s2.2}
Suppose now that the closed symmetric monoidal category $\sC$ is also a model category.
The category $\sC$ is called {\em a monoidal model category} \cite[Def. 4.2.6]{Hov} if the model and monoidal structures are compatible in the sense that the following three conditions hold.
\begin{enumerate}
\item The monoidal structure $\otimes : \sC \times \sC \rightarrow \sC$ is a Quillen bifunctor \cite[4.2]{Hov}, i.e., 
given two cofibrations $f,g\in \Cof$, $f\in \sC (U,V)$,
$g\in\sC (X,Y)$,
their pushout
$$
f \Box g:
(V\otimes X) \coprod_{U\otimes X} (U\otimes Y) \rightarrow V\otimes Y
$$
is a cofibration.
\item If one of the cofibrations $f$, $g$ is a trivial cofibration, then
$f \Box g$ is a trivial cofibration.
\item
For all cofibrant $X$ and cofibrant replacements of the monoidal unit
$$
\emptyset_{\unit}: \emptyset \xrightarrow{\Cof} {\unit}_{\Cof} \xrightarrow{f \ \Fib \WEq} \unit
$$
the maps
$$
f \otimes \Id_X : {\unit}_{\Cof} \otimes X \rightarrow \unit\otimes X
, \qquad
\Id_X \otimes f : X \otimes {\unit}_{\Cof}  \rightarrow X \otimes \unit
$$
are weak equivalences.
\end{enumerate}

Notice that condition~{(3)} holds automatically if $\unit$ is cofibrant.

The upshot of this definition is that the homotopy category $\HO (\sC)$ becomes a  closed symmetric monoidal category under the left derived tensor product $\otimes^L$ and the right derived internal homs ${\mathrm R}[-,-]$ and ${\mathrm R}\widetilde{[-,-]}$ with the monoidal unit
$\unit$ \cite[4.3.2]{Hov}.

\subsection{Induced model structures for modules and comodules}
\label{s2.3}

We would like to equip the category
$\sC_T$ with a left induced model structure and the category
$\sC^F$ with a right induced model structure. 
The forgetful functors to $\sC$ are denoted $G_T$ and $G^F$ respectively.
The maximal right (left) complementary class of a class of morphisms $\bX$
is denoted $\bX^\boxslash$ (${}^{\boxslash}\bX$ correspondingly). 
Let us define the classes of maps
\begin{equation}
\label{ind_cl}  
\Cof_{T} \coloneqq G_T^{-1}(\Cof), \
\WEq_{T} \coloneqq G_T^{-1} (\WEq), \
\Fib_{T} \coloneqq ( \Cof_{T} \cap \WEq_{T})^{\boxslash}, \
\end{equation}
$$
\Cof^{F} \coloneqq {}^{\boxslash} ( \Fib^{F} \cap \WEq^{F}), \ 
\WEq^{F} \coloneqq G^{F\, -1} (\WEq), \ 
\Fib^{F} \coloneqq G^{F\, -1} (\Fib).
$$
Even if the categories  $\sC_T$ and $\sC^F$ are complete and cocomplete (cf. Section~\ref{s1.2b}), these classes do not necessarily define model structures.
The following proposition gives some sufficient conditions. Further sufficient conditions are known (cf. \cite[Th. 5.8]{HeShi}, \cite[Th. 4.1]{ScSh}). 
\begin{prop}
  \label{ind_model}
  Suppose that $\sC$ is a closed symmetric monoidal model category such that the model category structure is accessible. Let $(T\dashv F)$ be an internally adjoint comonad-monad pair. 
  \begin{enumerate}
  \item If the category $\sC_T$ is locally presentable, then
      $\sC_T$ is complete and
     equation~\eqref{ind_cl} defines an accessible model structure on $\sC_T$, called (left)-induced.
  \item If the category $\sC^F$ is cocomplete, then
    $\sC^F$ is locally presentable and
    equation~\eqref{ind_cl} defines an accessible model structure on $\sC^F$, called (right)-induced.
    \end{enumerate}
\end{prop}
\begin{proof}
A locally presentable category is complete \cite[Cor. 1.28]{AdRo}.  Then part (1) follows immediately from \cite[Cor. 3.3.4]{HeKe}.

  The category $\sC^F$ admits small limits and colimits by
  our assumptions (cf. Section~\ref{s1.2b}).
Now, the functor $F:\sC \rightarrow \sC$ is a right adjoint, hence, accessible by \cite[Prop. 2.23]{AdRo}.  By \cite[Th. 1.20]{AdRo}, $\sC^F$ is accessible. Since $\sC^F$ is complete, it is locally presentable \cite[Cor. 2.47]{AdRo}.

The second statement in (2) follows from \cite[Cor. 3.3.4]{HeKe}.
\end{proof}

We finish the section with the following fact:
  \begin{cor}
    \label{ind_model_loc_pres}
    Suppose that, further to the conditions of Proposition~\ref{ind_model}, the category $\sC$ is locally presentable. Then the following statements hold.
    \begin{enumerate}
\item Equation~\eqref{ind_cl} defines an accessible (left-induced) model structure on $\sC_T$ and
  an accessible (right-induced) model structure $\sC^F$.
\item If $\sC$ is cofibrantly generated or right proper, with generating set of trivial cofibrations $\bJ$, and if the functor $G^F$ takes relative $F \bJ$-complexes to weak equivalences, then $\sC^F$ is also cofibrantly generated or right proper, respectively.  
 \end{enumerate}
 \end{cor}
   \begin{proof}
    The first statement follows from Proposition~\ref{ind_model} and Proposition~\ref{chief_presentable}.

By Proposition~\ref{chief_presentable}  $\sC^F$ is locally presentable. Thus, combined with our assumption on $G^F$, it follows that $\sC^F$ is cofibrantly generated by 
\cite[Th. 11.3.2]{Hir}. Since limits in $\sC^F$ are inherited from $\sC$, the model structure on $\sC^F$ is right proper.
\end{proof}

\subsection{Comodule-contramodule correspondence for model categories}
\label{s2.4}

Let us consider the following diagram of categories and the three pairs of $\sC$-enriched adjoint functors $(F\dashv G^F)$, $(G_T\dashv T)$ and $(L \dashv R)$ (cf. Theorem~\ref{adjoint_LR}).
\begin{equation}
  \label{categor_triangle}
\begin{tikzcd}[row sep=2.5em]
 &  \sC \arrow[dl,shift right,swap,"F"]\arrow[dr,shift right,swap,"T"]\\
\sC^F  \arrow[ru,shift right,swap,"G^F"]  \arrow[rr,shift left, "L"] &&  \sC_T\arrow[ll,shift left,"R"] \arrow[ul,shift right,swap,"G_T"]
\end{tikzcd}
\end{equation}
All these adjunctions are $\sC$-enriched.
Assuming that equation~\eqref{ind_cl} defines model structures,
the adjunctions $(F\dashv G^F)$ and $(G_T\dashv T)$ are Quillen adjunctions.
What about the third adjunction $(L \dashv R)$?

\begin{prob}
  \label{prob_QA}
  \begin{enumerate}
    \item Find necessary and sufficient conditions for the adjunction $(L \dashv R)$ to be a Quillen adjunction (and/or a Quillen equivalence) between the right-induced model category $\sC^F$ and the left-induced model category $\sC_T$.
\item  Investigate existence of other model category structures on $\sC^F$ and $\sC_T$ (or their co(completions)) under which the adjunction $(L \dashv R)$ is a Quillen adjunction or a Quillen equivalence.
\end{enumerate}  \end{prob}


\subsection{An answer for cartesian closed categories}
In this section 
we assume that $\sC$ is a cartesian closed category. This means that the monoidal product $\otimes$ in $\sC$ is the categorical product. It follows that $\sC$ is symmetric and the unit object $\unit$ is the terminal object. Similarly to the start of Section~\ref{s1.6}, all comonoids in such category are objects $C$ with the diagonal map $\Delta :C \rightarrow C \times C$.

Let a comonoid $\chf$ be a chief object of an internally adjoint comonad-monad pair on $\sC$. Similarly to Proposition~\ref{chf_sets_comodules}, $\sC_T$ is isomorphic to the overcategory (or slice category) $(\sC \!\downarrow\! \chf)$ (c.f. \cite{Hir2}):
\begin{equation} \label{slice_isomorphic}
(M,\rho : M\rightarrow T(M)) \leftrightarrow (M, \phi : M\rightarrow \chf) \ \mbox{ where } \ \rho=(\phi,\Id_M) \, . 
\end{equation}

\begin{prop} \label{slice_complete} 
The category $\sC_T$ is complete and cocomplete.
\end{prop}
\begin{proof}
The slice category of a complete category is complete \cite[IV.7, Th. 1]{MacMo}. Cocompleteness is immediate (Section~\ref{s1.2b}). 
\end{proof}

The left-induced model structure (see~\eqref{ind_cl}) on $\sC_T$ is, in fact, {\em induced}:
\begin{prop} \label{ind_model_str}   (cf. \cite{Hir2})
If $\sC$ is cofibrantly generated, then the following is a cofibrantly generated model structure on $\sC_T$: 
\begin{equation}
\label{ind_cartesian_cat}  
\Cof_{T} = G_T^{-1}(\Cof), \
\WEq_{T} = G_T^{-1} (\WEq), \
\Fib_{T} = G_T^{-1} (\Fib). 
\end{equation}
If $\sC$ is left or right proper, then so is $\sC_T$.  
\end{prop}
\begin{proof}
We identify $\sC_T$ with $(\sC \!\downarrow\! \chf)$. 
Since $\sC$ is a cofibrantly generated model category, so is $(\sC \!\downarrow\! \chf)$ under the model structure~\eqref{ind_cartesian_cat} \cite[Th. 1.5]{Hir2}.
This proves the first statement. 

The second statement is \cite[Th. 1.7]{Hir2}.
\end{proof}

We do not know any special description of $\sC^F$ in the cartesian case but the behaviour of the comodule-contramodule correspondence is distinctive.
\begin{prop}
\label{Quillen_adjunction_cartesian_cat}  
Suppose that $\sC$ is cartesian closed, the left-induced model structure exists on $\sC_T$ and the right-induced model structure exists on $\sC^F$.  
Then the pair $({L}\dashv R)$ is a Quillen adjunction.  
\end{prop}
\begin{proof}
We need to show that the functor $R: \sC_T \to {\sC^F}$ preserves fibrations and trivial fibrations. 
Let $f: (X, \phi_X) \to (Y, \phi_Y)$ be a (trivial) fibration in $\sC_T$. Since the model structure on ${\sC^F}$ is right-induced, we need to verify that $R f$ is a (trivial) fibration in $\sC$. Let us consider a commutative diagram in $\sC$
\begin{center}
\begin{tikzcd}
  U \arrow[r] \arrow{d}[swap]{\Cof\cap\WEq\ni (\mbox{ or }\Cof\ni)} 
    & R X \arrow[d, "R f"]  \\
  V  \arrow[ur, dashrightarrow, "h"] \arrow[r, "g"]
    & R Y
\end{tikzcd}
\end{center}
where the left down arrow is a trivial cofibration (correspondingly, a cofibration) in $\sC$. The diagonal filling $h$ has not been found yet. Since $R X$ is a subobject of $FX=[\chf,X]$,  we have the adjunct commutative diagram
\begin{center}
\begin{tikzcd}
    TU=U \times \chf \arrow[r] \arrow{d}[swap]{\Cof\cap\WEq\ni (\mbox{ or }\Cof\ni)} 
    & X \arrow[d, "f"]  \\
     TV= V  \times \chf \arrow[ur, dashrightarrow, "\hat{h}"] \arrow[r, "\hat{g}"]
    & Y
\end{tikzcd}
\end{center}
where the left down arrow is also a trivial cofibration (a cofibration) in $\sC$. Since the model structure on $\sC_T$ is induced, $f$ is a (trivial) fibration in $\sC$. Thus, there exists a diagonal filling $\hat{h}$, whose adjunct map $h:V\rightarrow [\chf,X]$ would be a diagonal filling of the first diagram if it were to factor through $R X\hookrightarrow FX$. This would imply that $R f$ is a (trivial) fibration, finishing the proof.

To prove the outstanding claim we need to show that $h$ equalises the pair of maps
$$\phi^T_{\chf,X},\psi^T_{\chf,X}: [\chf,X]\rightrightarrows [\chf,TX]=[\chf,X\times \chf]\cong [\chf,X]\times [\chf,\chf]\, ,$$
defined in Section~\ref{s1.5}.
The first components of these maps are equal so that we need to prove that
$$
(\phi^T_{\chf,X})_1 \circ h = (\psi^T_{\chf,X})_2 \circ h: [\chf,X]\rightrightarrows [\chf,\chf] \, .
$$
This follows from the fact that $g:V\rightarrow R Y$ equalises the similar maps for $Y$ and the commutativity of the following diagram: 
\begin{equation} \label{cartesian_diagram}
  \begin{tikzcd}
U \arrow[r] \arrow{d}[swap]{\Cof\cap\WEq\ni (\mbox{ or }\Cof\ni)} 
    & R X \arrow[d] \arrow[r]  & {[\chf ,X]} \arrow[d, "Ff"]\arrow[r]  & {[\chf , \chf]} \arrow[d, "\Id_{[\chf , \chf]}"]  \\
  V  \arrow[urr, dashrightarrow, "h"] \arrow{r}[swap]{g}
    & R Y\arrow[r]& {[\chf , Y]} \arrow[r, shift left]\arrow[r,shift right]  & {[\chf , \chf]}
\end{tikzcd}
\end{equation}
\end{proof}

For the pair $(L\dashv R)$ to be a Quillen equivalence, the maps
\begin{equation} \label{Quillen_equivalence_maps}
\eta_X : X \rightarrow R (L X) \rightarrow R ( (L X)_{\bF}), \ \
\epsilon_Z : L ((R Z)_{\bC}) \rightarrow L (R Z) \rightarrow Z\, 
\end{equation}
for all $X\in (\sC^F)_{\bC}$, $Z\in (\sC_T)_{\bF}$, 
derived from the unit and the counit of adjunction,
must be weak equivalences. 
For this to be true it suffices to localise at the classes of maps $\bA$ and $\bB$ as constructed below. First start with factorising the maps $\eta_X$ and $\epsilon_Z$:
$$\eta_X:X\xrightarrow{g_X\ \Cof} X^\prime \xrightarrow{\Fib\WEq} R ((L X)_{\bF})$$
$$\epsilon_Z :L ( (R Z)_{\bC}) \xrightarrow{k_Z\ \Cof\WEq} Z^\prime \xrightarrow{\ \Fib} Z.$$
Taking fibrant and cofibrant replacements $X'_\bF$ and $Z'_\bC$ of the objects $X'$ and $Z'$ respectively, we obtain maps:
$$ r_{X}: X \xrightarrow{g_X} X' \to X'_\bF \ \text{and}\ q_{Z}: Z'_\bC \to Z' \xrightarrow{k_Z} Z.$$
Factorising these gives us our desired classes:
\begin{equation}
\label{localisation_classes}  
\bA \coloneqq \{f_X \,\mid\, X\xrightarrow{\Cof \WEq} X^{''} \xrightarrow{f_X \ \Fib} X'_{\bF}\}, \
\end{equation}
$$
\bB \coloneqq \{h_Z \,\mid\, Z'_{\bC} \xrightarrow{\Cof} Z^{''} \xrightarrow{\Fib \WEq} Z\}.
$$
\begin{thm}
\label{Quillen_equivalence_cartesian_cat}  
  Let us make the following assumptions:
  \begin{enumerate}
  \item $\sC$ is a locally presentable category,
  \item $\sC$ is a cartesian closed monoidal model category,
  \item $\sC$ is a left and right proper model category,
  \end{enumerate}
  Then there exist a right Bousfield localisation $\RB_{\bA}(\sC^F)$ and a left Bousfield localisation $\LB_{\bB}(\sC_T)$, so that the co-contra correspondence $(L\dashv R)$ induces a Quillen equivalence between them. 
\end{thm}
\begin{proof}
  We engineer the localisation classes so that $(L\dashv R)$ would induce a Quillen equivalence. The only thing we need to check is that the localisations  exist.

  First, instead of the localisation classes we can use localisation sets because the categories $\sC_{T}$ and $\sC^{F}$ are locally presentable by Proposition~\ref{chief_presentable}. We define 
  $$
  \bA^{\flat} \coloneqq \{f_{Y_\bC}\in \bA \,\mid\, Y \mbox{ is in the generator}\},
  $$
  $$
  \bB^{\flat} \coloneqq \{h_{U_\bF}\in \bB \,\mid\, U \mbox{ is in the generator}\}.
$$
  These are sets of maps. If these maps are turned into weak equivalences, the adjunction units and counits for $Y$ and $U$ become isomorphisms in the homotopy categories. Recall that the Quillen adjunction $(L\dashv R)$ descends to a pair of adjoint functors between the homotopy categories $\HO (\sC^F)$ and $\HO (\sC_T)$.
  
  Observe that $Y$ belongs to the set of generating objects of $\sC^F$.
  The cofibrant resolutions $Y_\bC$
  form a set of generating objects of $\HO (\sC^F)$. Thus, the adjunction unit is an isomorphism for all objects in $\HO (\sC^F)$.
A similar argument shows that the adjunction counit is an isomorphism for all objects in $\HO (\sC_T)$. 

It remains to show the existence of the localisations.
Proposition~\ref{ind_model_str} yields that $\sC_T$ is a left proper combinatorial model category and so $\LB_{\bB^{\flat}}(\sC_T)$ exists. 
Similarly, 
all the conditions for existence of $\RB_{\bA^{\flat}}(\sC^F)$, stated in \cite[Rmk. 5.1.2]{Hir}, are met.

Finally, it is clear that 
$\LB_{\bB}(\sC_T)=L_{\bB^{\flat}}(\sC_T)$ and  
$\RB_{\bA}(\sC^F)=R_{\bA^{\flat}}(\sC^F)$.
\end{proof}

\subsection{Simplicial sets}
\label{SSets_2}
A good example of a category satisfying all conditions of Theorem~\ref{Quillen_equivalence_cartesian_cat} is the category $\sS$ of simplicial sets, briefly discussed in Section~\ref{SSets_1}, with respect to the classical (Quillen) model structure (cf. \cite[Def. 7.10.8]{Hir}). The category $\sS$ is locally presentable as it is a presheaf category \cite[1.46]{AdRo}, proper (\cite[Th. 13.1.13]{Hir})
 and cartesian closed.

Let $\chf=(\chf_n)\in\sS$, a comonoid under the diagonal map, be the chief of an internally adjoint comonad-monad pair $(T\dashv V)$. Let us summarise its comodule-contramodule correspondence:
  \begin{thm}\label{simplicial_set_CCC_strong}
\begin{enumerate}
\item    The adjoint pair $(L\dashv R)$ is a Quillen adjunction between ${\sS}_T$ and ${\sS}^F$.
\item    The adjoint pair $(L\dashv R)$ is a Quillen equivalence between the right Bousfield localisation $\RB_{\bA}(\sS^F)$ and the left Bousfield localisation $\LB_{\bB}(\sS_T)$.
\item All contramodules are cofibrant objects of $\sS^F$.
\item A comodule $(X,\phi)$ is a fibrant objects of $\sS_T \cong (\sS \!\downarrow\! \chf)$ if and only if $\phi:X\rightarrow\chf$ is a Kan fibration.
  \end{enumerate}
  \end{thm}
  \begin{proof} Statement~{(1)} is Proposition~\ref{Quillen_adjunction_cartesian_cat}.  
Statement~{(4)} is the definition.

It is clear that $\chf$ is  $\lambda$-presentable where $\lambda$ is a regular cardinal greater than the cardinality of the union $\cup_n \chf_n$. Thus, statement~{(2)} is Theorem~\ref{Quillen_equivalence_cartesian_cat}

Let $\Delta[n]\in\sS$ be the $n$-dimensional simplex. Observe that $F (\Delta[1])$ is a cylinder object in $\sC^F$. This yields the cylinder decomposition of the empty map 
$$
\emptyset_X : \emptyset \xrightarrow{\Cof^F} \Cyl (\emptyset \rightarrow X) \xrightarrow{\WEq^F}  X 
$$
for all $X\in\sC^F$. Since $\emptyset\times X = \emptyset$, the second map $\Cyl (\emptyset \rightarrow X) \rightarrow X$ must be the identity. This proves statement~{(3)}. 
\end{proof}

  Notice that $(L\dashv R)$ is not a Quillen equivalence between ${\sS}_T$ and ${\sS}^F$ even for ``nice'' simplicial sets $\chf$.
  There exist $\chf$-comodules $(X,\phi)$ such that the map of geometric realisations $|\phi|:|X|\rightarrow |\chf|$ has no continuous sections. It follows that $R X$ is empty. See Section~\ref{top_CCC} for further discussion.

\subsection{Positselski's answer}
\label{s4e}
Let $\sC=\KCh$ be the category of chain complexes over a commutative ring $\K$ with the standard monoidal structure and the Quillen model structure \cite[Th. 1.4]{BMR}, \cite[Th. 2.3.11]{Hov}.

A comonoid in $\KCh$ is a DG-coalgebra.
Since $\KCh$ is locally presentable, any DG-coalgebra is presentable.
By Proposition~\ref{chief_presentable},  $\KCh^F$ and $\KCh_T$ are complete, cocomplete and locally presentable categories. 

The Quillen model structure on $\KCh$ is compactly generated \cite[Th. 1.4]{BMR}, hence, accessible. Proposition~\ref{ind_model} yields the left-induced model structure $(\Cof_{T},\WEq_{T},\Fib_{T})$ on $\KCh_T$ and the right-induced model structure $(\Cof^{F},\WEq^{F},\Fib^{F})$ on $\KCh^F$. Positselski calls them {\em projective} and {\em injective} correspondingly. Since the category of chain complexes is not cartesian closed, neither Proposition~\ref{Quillen_adjunction_cartesian_cat}, nor Theorem~\ref{Quillen_equivalence_cartesian_cat} are applicable. This makes the following variation of Problem~\ref{prob_QA} interesting. 
\begin{prob}
  \label{prob_QA_B}
  Find necessary and sufficient conditions on the commutative ring $\K$ and the chief DG-coalgebra $\chf$ for the adjunction $(L \dashv R)$ to be a Quillen adjunction (and/or a Quillen equivalence) between the injective model category $\KCh^F$ and the projective model category $\KCh_T$.
\end{prob}

Instead of answering this question, Positselski gives an alternative answer to the part (2) of  Question~\ref{prob_QA}. He makes an additional assumption that 
\begin{equation}
  \label{Pos_assmp}
  \chf \mbox{ is } \K\mbox{-projective and }\K\mbox{ is of finite global dimension.}
\end{equation}
This assumption ensures that the categories $\KCh_T$ and $\KCh^F$ are abelian.
Positselski proves that under this assumption 
$\KCh_T$ admits a semi-projective model structure $(\Cof^{p}_T,\WEq^{p}_T,\Fib^{p}_T)$ \cite[9.1]{Pos2}
(the letter $p$ in the notation stands for Positselski),
while
$\KCh^F$ admits a semi-injective model structure $(\Cof_{p}^F,\WEq_{p}^F,\Fib_{p}^F)$ with the following properties \cite[Rmk. 9.2.2]{Pos2}:
\begin{enumerate}
\item $\Cof^{p}_T = \Cof_T$,  $\WEq^{p}_T \subseteq \WEq_T$, $\Fib^{p}_T \supseteq \Fib_T$,
\item $\Cof_{p}^F \supseteq \Cof^F$, $\WEq_{p}^F \subseteq \WEq^F$, $\Fib_{p}^F = \Fib^F$,
\item  The co-contra correspondence $(L \dashv R)$ is a Quillen equivalence between
  $(\KCh_T, \Cof^{p}_T,\WEq^{p}_T,\Fib^{p}_T)$ and
  $(\KCh^F, \Cof_{p}^F,\WEq_{p}^F,\Fib_{p}^F)$.
\end{enumerate}
A proof of this fact is only indicated in \cite{Pos2}. In our view, the model structures on $\KCh_T$ and $\KCh^F$ deserve a thorough investigation in the spirit of \cite{BMR}. For instance, there are indications that imposing the condition~\eqref{Pos_assmp} above is too strong. 
\begin{prob}
  \label{prob_QA_B1}
  For an arbitrary commutative ring $\K$ and a DG-coalgebra $\chf$, do there exist a semi-injective model category $\KCh^F$ and a semi-projective model category $\KCh_T$ that satisfy the three properties just above?
\end{prob}

\section{Topological Spaces}
\label{s6}

\subsection{A convenient category of topological spaces $\sW$}
\label{conv_top_spaces}
The category of topological spaces $\sT$ is not closed monoidal.
To remedy this issue, Steenrod suggested the notion of a convenient category \cite{Stn}.
The most common convenient category is the category $\sW$ of compactly generated weakly Hausdorff topological spaces, introduced by McCord \cite{Mcc}.
We follow a modern exposition by Schwede \cite[App. A]{Schw}.
Consider subcategories
$$
\sW \stackrel{\vi}{\hookrightarrow}
\sK \stackrel{\vi}{\hookrightarrow}
\sT\, 
$$
where $\sT$ is the category of topological $\sK$ is the category of compactly generated topological spaces. The embedding functors have adjoint functors the Kellification functor $\vk$ and the weak Hausdorffication functor $\vw$:
$$
\sW \xleftarrow{\vw}
\sK \xleftarrow{\vk}
\sT\, , \ \ ( \vi \dashv \vk) \, , (\vw \dashv  \vi ) \, .
$$
We use a subscript to denote the category in which a construction is taking place:
\begin{equation} \label{products}
X \times Y \coloneqq X\times_{\sW} Y
= X\times_\sK Y 
= \vk (X\times_{\sT} Y) \, , \,
\end{equation}
$$
\prod X_n = {\prod}_{\sK} X_n = \vk ( {\prod}_{\sT} X_n) \, . 
$$
No subscript means that the construction is taking place in the default category $\sW$.
Formula~\eqref{products} tells us how the products in different categories relate.
A similar relation holds for arbitrary limits: 
$$
\varprojlim H = 
{\varprojlim}_{\sK} H =  \vk ({\varprojlim}_{\sT} H) \, .
$$
On the other hand, the coproducts are the same in all three categories:
$$
\coprod X_n = {\coprod}_{\sK} X_n = {\coprod}_{\sT} X_n \, . 
$$
Since quotients of weakly Hausdorff spaces are no longer weakly Hausdorff, the relation for colimits is this:
$$
\varinjlim H = 
\vw ({\varinjlim}_{\sK} H) =  \vw ({\varinjlim}_{\sT} H) \, .
$$
Both categories $\sW$ and $\sK$ are closed symmetric monoidal categories \cite[A.22, A.23]{Schw} with products $X\times Y$ and $X\times_{\sK} Y$ and internal homs 
$$
[X,Y]_{\sW} = \vk (C(X,Y))= \vk(C^\prime (X,Y)), \  
[X,Y]_{\sK} = \vk(C^\prime (X,Y)),
$$
where $C(X,Y)= C^\prime (X,Y) = \sT (X,Y)$ is the set of continuous functions
\linebreak
$X \to Y$.
The difference is the topology. The space $C(X,Y)$ carries the compact open topology, 
while $C^\prime (X,Y)$ is equipped
with the {\em modified} compact open topology. 
The basis of the latter is given by sets of the form 
$$N(h,U) \coloneqq \{ f: X \to Y\ |\ f\ \text{is continuous}, f(h(K)) \sset U\},$$
where $U$ is open in $Y$, $K$ is compact and $h: K \to X$ is a continuous map. 
Notice that if $X$ is weakly Hausdorff, then $h(K)$ is closed and thus compact. So the two topologies on $\sT (X,Y)$ coincide in this case.



\subsection{Homotopy theory in $\sW$}
The {\em Quillen model structure} on $\sW$ is defined as follows.
\begin{itemize}
\item[$\WEq$,]
{\bf weak equivalences.}  These are the maps $f: X \to Y$ satisfying
\begin{itemize}
\item[(i)] $f$ induces an isomorphism of sets $\pi_0(X) \xrightarrow{\cong} \pi_0(Y)$,
\item[(ii)] and for any $x \in X$ and $n \geq 1$ the induced homomorphism $f_*: \pi_n(X, x) \to \pi_n(Y, f(x))$ is an isomorphism of groups.
\end{itemize} 
\item[$\Fib$,]
{\bf fibrations.} The fibrations are the Serre fibrations, that is, those maps $p : E \to B$ which have the homotopy lifting property with respect to any CW-complex.
\item[$\Cof$,]
{\bf cofibrations.}  The cofibrations are the maps $f : X \to Y$ which are retracts of a map $f' : X \to Y'$, where $Y'$ is a space obtained from $X$ by attaching cells.
\end{itemize}

%

Note that $\sW$ with the Quillen model structure is a cofibrantly generated model category with a set of generating cofibrations  
\begin{equation} \label{gencof}
\bI =\{ S^{n-1} \to D^{n}\ |\ n \geq 0\} \, , 
\end{equation}
where $S^n$ is an $n$-sphere and $D^n$ is an $n$-disc, 
and a set of generating trivial cofibrations 
\begin{equation} \label{gentrivcof}
\bJ =\{ D^{n} \x \{0\} \to D^{n} \x [0,1] \ |\ n \geq 0\}.
\end{equation}


\subsection{Cospaces}
\label{cospaces}
This is the name we will use for comodules in $\sW$.

Pick an internally adjoint comonad-monad pair $(T\dashv F)$ and its chief comonoid $\chf\in\sW$, with the diagonal as a comultiplication.
Consider an object $(X,\phi_X)$ of $(\sW \!\downarrow\! \chf)$.
Here $X$ is an object of $\sW$ and $\phi_X : X \to \chf$ is a map in $\sW$.
A morphism $f: (X, \phi_X) \to (Y,\phi_Y)$ is a map $f : X \to Y$ over $\chf$, in the sense that $\phi_X=\phi_Yf$.
Now let
$$
[X,Y]_{\chf} \subseteq [X,Y]_{\sW}
$$
be the subset of maps over $\chf$. 
(c.f. \cite{Hir2}). 


\begin{prop}
  \label{closed}
  $[X,Y]_{\chf}$ is a closed subset of $[X,Y]_{\sW}$.
\end{prop}
\begin{proof}
  Pick $f\in  [X,Y]_{\sW}\setminus [X,Y]_{\chf}$.
  There exists $x\in X$ such that $\phi_Y (f(x)) \neq \phi_X (x)$.
  Since $\phi_Y^{-1}(\phi_X (x))$ is closed, we can choose an open set $U\subseteq Y$ such that $f(x)\in U$ and $U\cap \phi_Y^{-1}(\phi_X (x))=\emptyset$.
 Then $f\in N(\{ x\},U)\subseteq  [X,Y]_{\sW}\setminus [X,Y]_{\chf}$ so that $[X,Y]_{\sW}\setminus [X,Y]_{\chf}$ is open and $[X,Y]_{\chf}$ is closed.
\end{proof}

It follows that $[X,Y]_{\chf}$ with the induced topology belongs to $\sW$. This makes the category $(\sW \!\downarrow\! \chf)$ enriched in $\sW$. 

The isomorphism of categories~\eqref{slice_isomorphic} between $(\sW \!\downarrow\! \chf)$ and $\sW_T$ for the comonad $TX=X\times {\chf}$ is enriched in $\sW$. From now on we identify $\sW_T$  with $(\sW \!\downarrow\! \chf)$ and call its objects {\em cospaces}. 

By Proposition~\ref{slice_complete} $\sW_{T}$ is complete and cocomplete. By Proposition~\ref{ind_model_str}, there exists a Quillen induced model structure on $\sW_{T}$.

\subsection{Contraspaces}
The cospaces reduce to something conceptually simple.
At the moment we do not know any conceptually simpler definition of a contraspace other than the general one --
{\em a contraspace} is a contramodule in $\sW$ or a space $X$ equipped with a map $\theta_X: [{\chf}, X]_{\sW} \to X$ satisfying the usual properties.

The monad $FX= [{\chf},X]_{\sW}$, 
is defined by the diagonal comonoid $({\chf},\Delta_{\chf})$.
By Proposition~\ref{C_TF enriched}, $\sW^{F}$ is a category enriched in $\sW$. 
As before, its enriched hom is denoted by $[X,Y]^F$.

To understand the space $[X,Y]^F$, we consider the subset
$$
[X,Y]^{\chf} \subseteq [X,Y]_{\sW}
$$
that consists of contramodule maps. Note that this subset is the ordinary hom $\sW^F (X,Y)$. We equip $[X,Y]^{\chf}$ with the subspace topology.

\begin{prop}
  \label{closed2}
  \begin{enumerate}
\item $[X,Y]^{\chf}$ is a weakly Hausdorff space.
\item If $Y$ is Hausdorff, then $[X,Y]^{\chf}$ is a closed subset of $[X,Y]_{\sW}$. Consequently, $[X,Y]^{\chf}\in \sW$.
    \end{enumerate}
\end{prop}
\begin{proof}
Any subspace of $[X,Y]_{\sW}$ is weakly Hausdorff \cite[Prop. A4(i)]{Schw}. This proves (1).

To show (2), start with picking $f\in  [X,Y]_{\sW}\setminus [X,Y]^{\chf}$.
There exists $g\in [{\chf},X]_{\sW}$ such that $\theta_Y (fg) \neq f(\theta_X (g))$.
Since $Y$ is Hausdorff, we can find non-intersecting open sets $U,V\subseteq Y$ such that $\theta_Y (fg)\in U$ and  $f(\theta_X (g))\in V$.
Then $f$ belongs to the open set
$r_g^{-1} (\theta_Y^{-1} (U)) \cap N(\{\theta_X (g)\}, V)$
where 
$r_g^{-1} (\theta_Y^{-1} (U))$ is the inverse image
of the open set $\theta_Y^{-1} (U) \subseteq [{\chf},Y]_{\sW}$ 
under the continuous map
$$
r_g : [X,Y]_{\sW} \rightarrow [{\chf},Y]_{\sW}, \ \ h \mapsto hg \, .
$$
Notice that no $h\in r_g^{-1} (\theta_Y^{-1} (U)) \cap N(\{\theta_X (g)\}, V)$
can be a ${\chf}$-contramodule map since 
$\theta_Y (hg)\in U$ and  $h(\theta_X (g))\in V$.
Hence, $[X,Y]_{\sW}\setminus [X,Y]^{\chf}$ is open and $[X,Y]^{\chf}$ is closed.

Finally, a closed subspace of a space in $\sW$ is in $\sW$ \cite[Prop. A5(i)]{Schw}. 
\end{proof}

Armed with this proposition, we can understand $[X,Y]^F$ now. A proof is left to the reader. 
\begin{cor}
There exists a natural homeomorphism between $[X,Y]^F$ and $\vk ([X,Y]^{\chf})$.
\end{cor}


By Section~\ref{s1.2b}
$\sW^{F}$ is complete and
inherits limits from $\sW$.

\begin{prop} If ${\chf}$ is connected, then $\sW^{F}$ 
inherits coproducts from $\sW$.
\end{prop}
\begin{proof}
  Let $X=\coprod_n (X_n,\theta_n)$ be a coproduct in $\sW$ of a family of contraspaces $(X_n,\theta_n)\in \sW^F$. Since ${\chf}$ is connected, a continuous function $f:{\chf}\rightarrow X$ takes values in one particular $X_{n_0}$. This enables us to define the contramodule structure on $X$ by
  $\theta_X (f)\coloneqq \theta_{n_0}(f)$ or 
  $$
  \theta_X : [{\chf},X]_{\sW} \xrightarrow{\cong} \coprod [{\chf},X_n]_{\sW} \xrightarrow{\coprod \theta_n } \coprod X_n = X \, .
  $$
This is a coproduct in  $\sW^F$: the universal property is immediate. \end{proof}

A category with coproducts is cocomplete if and only it admits coequalisers. However, coequalisers are not inherited from $\sW$, even for a connected ${\chf}$.
\begin{lemma} \label{B_presentable}
A space $X$ is presentable if and only if $X$ is discrete.
\end{lemma}
\begin{proof}
  If $X$ is discrete, then $[X,-]_{\sW}$ commutes with $|X|$-directed colimits.

 Suppose that $X$ is not discrete. Let $X_d$ denote the set $X$ with the discrete topology. Given a  limit ordinal $\Lambda$ and $\Omega\in \Lambda$, let
$X_\Omega \coloneqq X^{\Lambda}$ as a set and 
  $X_\Omega \coloneqq (\prod_{\Upsilon < \Omega} X) \times (\prod_{\Upsilon \geq \Omega} X_d)$ as
  a topological space. The colimit $\clm (\ldots X_\Omega \xrightarrow{\Id} X_{\Omega +1}\ldots )$
  is $X^\Lambda$ as a topological space but the diagonal map $\Delta : X\rightarrow X^{\Lambda}$
  does not factor through any $X_\Omega$.
\end{proof}

We define {\em a subcontraspace} of $(X,\theta_X)$ as a subset $Y$ of $X$ such that $\theta_X (f)\in Y$ for any continuous function $f:{\chf}\rightarrow Y$.
We denote a subcontraspace by $Y\leq X$.

Consider the subspace topology on $Y\leq X$. Clearly, $Y\in \sK$. 
Since $\sW$ is closed under closed subsets \cite[A5]{Schw}, if $Y$ is closed, $Y$ is a contraspace itself. In general, $\vk (Y)$ is a contraspace because $\sK({\chf},Y)=\sW({\chf},\vk(Y))$ due to the adjunction $( \vi \dashv \vk)$. Thus, $\theta_Y$ is obtained by restricting $\theta_X$ to $[{\chf},\vk (Y)]_{\sW}\subseteq [{\chf},X]_{\sW}$. The continuity of $\theta_Y$ is clear.

The following lemma is obvious:
\begin{lemma} \label{intersection}
An arbitrary intersection of subcontraspaces is a subcontraspace.
\end{lemma}
In particular, the empty set is a subcontraspace with the structure map
$\Id_{\emptyset} : [{\chf},\emptyset]_{\sW} = \emptyset \rightarrow \emptyset$. 
Lemma~\ref{intersection} allows us to define, given a subset $Z\subseteq X$ of a contraspace $X$,
the subcontraspace generated by $Z$: 
$$
Z^{\chf} \coloneqq \bigcap_{Z \subseteq Y \leq X} Y\, .
$$
Let us describe $Z^{\chf}$ constructively. 
For an ordinal $\Omega$ we define the following sets by transfinite recursion: 
$$
Z_0\coloneqq Z, \ \
Z_{\Omega}\coloneqq 
\begin{cases} \theta_X ([\chf, Z_{\Omega -1}]_{\sT}) &\mbox{if } \Omega \mbox{ is a successor ordinal,} \\
  \bigcup_{\Upsilon \leq \Omega} Z_{\Upsilon}& \mbox{if } \Omega \mbox{ is a limit ordinal.} \end{cases} 
$$
\begin{prop} \label{ZB_description}
  If $\Omega$ is a $|{\chf}|$-filtered ordinal, then
$Z^{\chf} = Z_{\Omega}$.  
  \end{prop}
\begin{proof}
  The inclusion $Z^{\chf} \supseteq Z_{\Omega}$ is obvious.

To prove the opposite inclusion, we need to show that $Z_{\Omega}$ is a subcontramodule.   
A continuous function $f:{\chf}\rightarrow Z_\Omega$ corestricts
to a function $f|^{ Z_{\Upsilon}} :{\chf}\rightarrow Z_{\Upsilon}$ for some $\Upsilon <\Omega$ because $\Omega$ is $|{\chf}|$-filtered.
Thus, $\theta_X(f) = \theta_X (f|^{ Z_{\Upsilon}}) \in Z_{\Upsilon +1}\subseteq Z_{\Omega}$.
\end{proof}

While ${\chf}$ is not presentable in general (Lemma~\ref{B_presentable}), the proof of Proposition~\ref{ZB_description} uses the fact that $[{\chf},-]_{\sW}$ commutes with special colimits (cf. \cite[Lemma 2.4.1]{Hov}). This can be sharpened to prove the following theorem.

\begin{thm} \label{TB_cocomplete}
  The category $\sW^{F}$ is cocomplete.
  \end{thm}
\begin{proof}
  Let $S:\sD\rightarrow \sW^{F}$ be a small diagram, $V$ its colimit in $\sW$.
  Hence, given a cocone $\Psi_X: SX\rightarrow Y$, $X\in D$ in $\sW^{F}$,
  we have a unique mediating morphism $\Psi^{\sharp}: V\rightarrow Y$ in $\sW$.

Clearly, the cocone factors through the subcontramodule, generated by the image of $\Psi^{\sharp}$:
  $$
  \Psi_X: SX\xrightarrow{\Phi_X} (\Psi^{\sharp} (V))^{\chf} \hookrightarrow Y \, .
  $$
The explicit construction in Proposition~\ref{ZB_description} gives an upper bound $\Omega$ on the cardinality of $(\Psi^{\sharp} (V))^{\chf}$. It depends on $|{\chf}|$ and $|V|$ but does not depend on $|Y|$.

Let us consider a category $\sD^{\ast}$, whose objects are cocones $\Psi_X: SX\rightarrow Y$ in $\sW^{\chf}$ with $|Y|<\Omega$.
The morphisms from $\Psi_X: SX\rightarrow Y$ to $\Phi_X: SX\rightarrow Z$ are such morphisms $f\in \sW^{\chf}(Y,Z)$ that $f\Psi_X=\Phi_X$ for all $X\in D$.
Since the cardinalities of the cocone targets in $\sD^{\ast}$ are bounded above,  the skeleton $\sD^{\ast}_0$ of $\sD^{\ast}$ is a small category.
Then
  $$
  S^{\ast}: \sD^{\ast}_0\rightarrow \sW^{\chf}, \ \ (\Psi_X: SX\rightarrow Y) \mapsto Y
  $$
  is a small diagram, whose limit ${\varprojlim}S^{\ast}$ is the colimit ${\varinjlim}S$.
      \end{proof}

We finish this section by right-inducing the Quillen model structure to $\sW^{F}$. 
It does not follow from Proposition~\ref{ind_model} because $\sW$ is not accessible.

\begin{prop}
There exists a Quillen right-induced model structure on ${\sW}^{F}$, defined by equations~\eqref{ind_cl}.
This structure is right proper.
\end{prop}
\begin{proof}
Since the Quillen model structure on $\sW$ is cofibrantly generated, a right induced model structure on $\sW^{F}$ exists if (cf. \cite[Th. 11.3.2]{Hir})
\begin{enumerate}
\item $F(\bI)$ and $F (\bJ)$ permit the small object argument
\item and the forgetful functor $G^F$ takes relative $F(\bJ)$-complexes to weak equivalences,
\end{enumerate}
where $\bI$ and $\bJ$ are the sets of generating cofibrations and generating trivial cofibrations as defined in (\ref{gencof}) and (\ref{gentrivcof}) respectively.
The second statement is obvious because the inclusions in
$$F(\bJ) =\{ [ \chf, D^{n} \x \{0\}] \to [\chf , D^{n} \x [0,1]\, ] \ |\ n \geq 0\}$$
admit deformation retracts. Hence, relative $F(\bJ)$-complexes are weak equivalences topologically.

The first statement holds because relative $F(\bI)$-complexes and relative $F(\bJ)$-complexes are topological inclusions and every topological space is small relative to the inclusions \cite[Lemma 2.4.1]{Hov}. 

The model structure described above is cofibrantly generated \cite[Th. 11.3.2]{Hir}. Since the model structure on $\sW$ is right proper, then so is the right-induced model structure on $\sW^{F}$. 
\end{proof}

\subsection{Topological comodule-contramodule correspondence}
\label{top_CCC}
Since $\sW$ is cartesian closed, the  pair $({L}\dashv R)$ is a Quillen adjunction by Proposition~\ref{Quillen_adjunction_cartesian_cat}. 
An analogue of Theorem~\ref{Quillen_equivalence_cartesian_cat} encounters set-theoretic difficulties. We can sweep them under the carpet and have the following result with an identical proof. 
\begin{prop}
  \label{top_CCC_Grothendieck}
Suppose that all topological spaces are subsets of a Grothendieck universe.  
Then there exist a right Bousfield localisation $\RB_{\bA}(\sW^{F})$ and a left Bousfield localisation $\LB_{\bB}(\sW_{T})$, where the sets $\bA$ and $\bB$ are defined similarly to classes in \eqref{localisation_classes}, 
so that the co-contra correspondence $(L\dashv R)$ induces a Quillen equivalence between the localisations. 
\end{prop}

Let $\chf=S^2$ be the 2-sphere. As a cospace, consider the 3-sphere $S^3$ with the Hopf fibration $\phi : S^3 \rightarrow S^2$. The cospace $(S^3,\phi)$ is fibrant, yet $R S^3=\emptyset$. This shows that $(L\dashv R)$ in Proposition~\ref{top_CCC_Grothendieck} is not a Quillen equivalence between $\sW^{F}$ and $\sW_{T}$. This example suggests some ``local'' version of the functor $R$ (using local sections as in the sheaf of sections) may still be an equivalence. 

Another instructive example is the 1-sphere $\chf=S^1$ and the figure-8 cospace $(X=S^1\vee S^1, \phi_X = \Cnt\vee\Id_{S^1})$. Clearly, $R X=\{\Id\}$ is the one-element set and $L R X=\chf$. Taking local sections does not help: local sections near the singular point are not going to see the collapsing loop in $X$. On the other hand, the collapsing loop will be ``seen'' by the local sections of the fibrant replacement $X_{\Fib}$. These phenomena deserve further investigation.

%

%

\subsection{Relation to simplicial sets}
Most of the current chapter equally applies to the category $\sK$ of compactly generated spaces, not only to $\sW$. An advantage of $\sK$ is its direct relation to the category of simplicial sets:
there is a Quillen equivalence between simplicial sets and topological spaces \cite[Th. 3.6.7]{Hov}
\begin{equation} \label{geom_simp}
(|-| \dashv \ssS )\, , \ \SC : \sS \rightleftarrows \sK : |-| \, 
\end{equation}
where $|Q_\bullet |$ is the geometric realisation of a simplicial set $Q_\bullet$ and $\SC(Y)_n = \sK (\Delta [n] , Y)$ is the singular complex of a topological space $Y$. 
Let $\chf_\bullet = (\chf_n)\in\sS$, $\chf=|\chf_\bullet|\in\sW$, $\widehat{\chf}_\bullet = \SC(\chf)\in\sS$, considered as comonoids in their categories. We denote the corresponding comonad-monad adjoint pairs by $(T\dashv F)$, $(T\dashv F)$ and $(\widehat{T}\dashv \widehat{F})$.

In light of the isomorphism of categories~\eqref{slice_isomorphic}, we identify the overcategories with the comodule categories. The functors~\eqref{geom_simp} and the induction (Proposition~\ref{right_adjoint_slice}) give rise to the following functors:
  $$
|-| : \sS_{T} \rightarrow \sK_{T}, \
\SC : \sK_{T} \rightarrow \sS_{\widehat{T}}, \ 
\sInd\!\!\downarrow : \sS_{\widehat{T}} \rightarrow \sS_{T} .$$
Similarly, we can use the functors~\eqref{geom_simp}.
The induction functor from Proposition~\ref{right_adjoint_F} can be applied  levelwise to some but not all simplicial contrasets (see Section~\ref{SSets_1}). We expect that the induction exists in general. These considerations yield the functors between the contramodule categories: 
$$
|-| : \sS^F \rightarrow \sK^F, \
\SC : \sK^F \rightarrow \sS^{\widehat{F}}, \
\sInd_{\bullet}^F  : \sS^{\widehat{F}} \rightarrow \sS^F.
$$  
We can package all these functors in the following conjectural worldview of the relation between the topological and the simplicial comodule-contramodule correspondences:
\begin{conj}\label{topological_simplicial_CCC_conj}
For any simplicial set $\chf$ there exists a commutative (in an appropriate sense) square of categories and Quillen adjunctions 
\begin{equation}
  \label{categor_square}
\begin{tikzcd}[row sep=2.5em]
  \sS^F  \arrow[d,shift right,swap,"|-|"]  \arrow[rr,shift left, dashrightarrow, "L"] 
     &
        &  \sS_T\arrow[ll,shift left, dashrightarrow,"R"] \arrow[d,shift right,swap,"|-|"] \\
  \sK^F  \arrow[u,shift right,swap,"\sInd_{\bullet}^F \circ{\SC}"]  \arrow[rr,shift left, dashrightarrow, "L"]
   & 
      &  \sK_T\arrow[ll,shift left, dashrightarrow, "R"] \arrow[u,shift right,swap,"\sInd\!\downarrow\circ{\SC}"]
\end{tikzcd}
\end{equation}
where the left adjoint functors are either on top or on the left
and the vertical solid arrows are Quillen equivalences.
\end{conj}

\subsection{Topological fact}
We finish the paper with a useful fact about the topological co-contra correspondence that does not follow from the general framework of model categories. 
\begin{prop} \label{easy_top_fact}
  Suppose $X,Y\in(\sW_T)_{\Fib}$ are CW-complexes.
If  $f\in\WEq_{T}(X,Y)$, then $R f\in\sW^F(FX, FY)$ and $Ff \in \sW([\chf,X],[\chf,Y])$ are homotopy equivalences.
\end{prop}
\begin{proof}
  By Whitehead Theorem, $f$ is a homotopy equivalence.
  Moreover, $f$ is a fibrewise homotopy equivalence \cite[7.5]{May}.
The rest of the argument is clear.
\end{proof}

In particular, $R f\in\WEq^F(R X, R Y)$.
We would like to refine Proposition~\ref{easy_top_fact}, replacing the CW-complex condition on $X$ and $Y$ with a condition on $\chf$. 

We need a standard topological lemma, which we could not find in the literature. Let $X$, $Y$ be connected topological spaces in $\sW$, $f\in\sW(X,Y)$. 
If $A\in\sW$ is another topological space, we write $f_A : \sW(A,X) \to \sW(A,Y)$ for the map of function spaces defined by composition with $f$ (cf. Section~\ref{conv_top_spaces}).  Next fix a map $\alpha : A \to X$ that will be a base point for $\sW(A,X)$.  As a base point for $\sW(A,Y)$ we use the map $\beta = f \circ \alpha$ so that $f_A : \sW(A,X) \to \sW(A,Y)$ is a map of pointed spaces.
  
  \begin{lemma} \label{Standard_top_fact}
Suppose that $A$ is a CW-complex of finite type and $f$ is a weak homotopy equivalence.  Then $(f_A)_n : \pi_n(\sW(A,X), \alpha) \to \pi_n(\sW(A,Y), \beta)$ is an isomorphism for all $n \geq 1$.
\end{lemma}
\begin{proof} The first step in the proof is to show that the result is true for the sphere $A = S^n$ where $n \geq 1$.  In this case the space $\sW(S^n, X)$ is usually denoted by $\Lambda^n(X)$.  Choose a base point for $S^n$.  Evaluating maps at the base point gives us a map $\Lambda^n(X) \to X$.  This map is a fibration and the fibre over $x \in X$ is the space $\Omega^n_x(X)$, the $n$-fold iterated based loop space of $X$, with base point $x$.  The map $f$ now gives a map of fibrations: 
$$
\begin{CD}
\Lambda^n(X) @>>>  \Lambda^n(Y)\\
@VVV @VVV \\
X @>>> Y
\end{CD}
$$
The homotopy groups of $\Omega^n_x(X)$ are given by $\pi_{k}(\Omega^n_x(X)) = \pi_{k+n}(X,x)$ for $k \geq 0$ and trivial for $k < 0$.  Under this identification, the map of homotopy groups $\pi_{k}$ induced by the map 
$$
\Omega^n_x(f) : \Omega^n_x(X) \to \Omega^n_{f(x)}(Y)
$$
is just 
$$
f_{k+n} : \pi_{k+n}(X,x) \to \pi_{k+n}(Y, f(x)).
$$
Since $f_*$ is a weak homotopy equivalence, it follows that the map of fibrations $\Lambda^{n}(X) \to \Lambda^{n}(Y)$ defines isomorphisms on the homotopy groups of the fibres.  Since $f$ is a weak homotopy equivalence this map of fibrations defines an isomorphism on the homotopy groups of the base spaces.  A standard five lemma argument shows that it, therefore, gives an isomorphism on the homotopy groups of the total spaces.

The second step is to extend the result to finite CW-complexes by induction on the number of cells. Assume that the map $(f_A)_* : \pi_n(\sW(A,X), \alpha) \to \pi_n(\sW(A,Y), \beta)$ is an isomorphism for $n \geq 1$. Now replace $A$ by $B = A \cup_{\psi} D^{p+1}$ with $\psi \in \sW(S^p , A)$.  This gives a cofibration sequence
$$ 
A \to B \rightarrow S^{p+1}.   
$$
Applying $\sW(-, X)$ and $\sW(-, Y)$ to this cofibration sequence and using the map $f : X \to Y$, leads to the following commutative diagram: 
$$
\begin{CD}
\sW(A,X) @<<< \sW(B,X) \\
@VVV @VVV \\
\sW(A,Y) @<<< \sW(B,Y).
\end{CD}
$$
The horizontal arrows are fibrations. The fibres of the top map are copies of $\sW(S^{p+1}, X)$. The fibres of the bottom one are copies of $\sW(S^{p+1}, Y)$.
By assumption, this map of fibrations induces an isomorphism on the homotopy groups of the base spaces,  and by the first step it induces an isomorphism on the homotopy groups of the fibres.  It follows from the five lemma that it induces isomorphisms on the homotopy groups of the total spaces.

The final step is to extend the result to a  CW-complex of finite type.  Let $A^{n}$ be the $n$-skeleton of $A$, $i_n : A^{n} \to A^{n+1}$ the inclusion.   Then $A$ is the direct limit of the $A^{n}$ and each of the inclusions $i_n$ is a cofibration.  It follows that $\sW (A,X)$ is the inverse limit of the sequence of maps $\sW(A^{n+1}, X) \to \sW(A^{n})$ induced by $i_n$. Since each of the maps $i_n$ is a cofibration, the maps in the inverse system are fibrations.  Now suppose $f : X \to Y$ is a weak equivalence.  We have proved that for each $n$ the map $f_{A^n} : \sW(A^n, X) \to \sW(A^n,Y)$ is a weak homotopy equivalence.  The map $f_A : \sW(A,X) \to \sW(A,Y)$ is the map of inverse limits defined by the sequence $f_{A^n}$. Hence, $f_A$ is also a weak homotopy equivalence \cite[Th. 2.2]{Hir3}.
\end{proof}

Given a topological space $X$ and a point $s\in X$, by $X_s$ we denote the connected component of $X$ that contains $s$. A map $f\in\sW(X,Y)$ yields a map $f_s\in\sW(X_s,Y_{f(s)})$ between components. 
\begin{thm}\label{last_thm}
Let $\chf$ be a CW-complex of finite type.
Suppose that  $(X,\phi),(Y,\psi)\in(\sW_T)_{\Fib}$ are fibrant cospaces and $s\in R X$. 
If $f\in\WEq_T (X,Y)$ is a weak homotopy equivalence, then the map  $R f_s$ is also a weak homotopy equivalence.
\end{thm}
\begin{proof}
Consider a part of the commutative diagram (~\ref{cartesian_diagram}):
\begin{center}
  \begin{tikzcd}
R X_s \arrow[d, "R f_s"] \arrow[r, "i"]  & {[\chf ,X]_s=FX_s} \arrow[d, "Ff_s"]\arrow[r, "\phi_{\ast}"]  & {[\chf , \chf]_{\Id}} \arrow[d, "\Id_{[\chf , \chf]}"]  \\
R Y_{fs}\arrow[r, "j"]& {[\chf , Y]_{fs}=FY_{fs}} \arrow[r,"\psi_{\ast}" ] & {[\chf , \chf]_{\Id}}
\end{tikzcd}
\end{center}
Since both $\phi$ and $\psi$ are fibrations, both $\phi_{\ast}=[\Id_{\chf} , \phi]$ and $\psi_{\ast}$ are also fibrations. Moreover, $R X_s$ is the fibre of $\phi_{\ast}$ over the identity and $R Y_{fs}$ is the fibre of $\psi_{\ast}$ over the identity.
All the spaces in the diagram have chosen base points.
This yields a map from the homotopy exact sequence of $\phi_{\ast}$ to the homotopy exact sequence of $\psi_{\ast}$.

The map of the base spaces is the identity: it induces the identity of homotopy groups.
By Lemma~\ref{Standard_top_fact}, the map of total spaces induces an isomorphism of homotopy groups.  The five lemma tells us that it induces an isomorphism on the homotopy groups of the fibres.
\end{proof}

If one shows $\pi_0(R f)$ is an isomorphism, then Theorem~\ref{last_thm} ensures that $R f$ is a weak homotopy equivalence. Such a proof would involve Topological Obstruction Theory and may require additional assumptions on $\chf$.  

Theorem~\ref{last_thm} is an indication that the co-contra correspondence is full of topological mysteries, waiting to be uncovered.

\

{\bfseries On behalf of all authors, the corresponding author states that there is no conflict of interest.}

\end{document}